\documentclass[reqno]{amsart}

\usepackage{enumitem}
\usepackage{amssymb,amsmath,mathtools}
\usepackage{hyperref}
\usepackage{microtype}
\usepackage{bbm}

\newcommand*{\mailto}[1]{\href{mailto:#1}{\nolinkurl{#1}}}
\newcommand{\arxiv}[1]{\href{http://arxiv.org/abs/#1}{arXiv:#1}}
\newcommand{\doi}[1]{\href{http://dx.doi.org/#1}{doi: #1}}
\newcommand{\msc}[1]{\href{http://www.ams.org/msc/msc2010.html?t=&s=#1}{#1}}

\newtheorem{theorem}{Theorem}[section]
\newtheorem{lemma}[theorem]{Lemma}
\newtheorem{corollary}[theorem]{Corollary}
\newtheorem{remark}[theorem]{Remark}

\newcommand{\R}{{\mathbb R}}
\newcommand{\N}{{\mathbb N}}
\newcommand{\Z}{{\mathbb Z}}
\newcommand{\C}{{\mathbb C}}
\newcommand{\bS}{{\mathbb S}}
\newcommand{\hr}{{\mathfrak H}}

\newcommand{\cRH}{{\mathcal R}_{H}}
\newcommand{\cRl}{{\mathcal R}_{l}}
\newcommand{\id}{{\mathbbm{1}}}
\newcommand{\OO}{{\mathcal O}}

\newcommand{\loc}{\mathrm{loc}}

\newcommand{\I}{\mathrm{i}}
\newcommand{\E}{\mathrm{e}}
\newcommand{\vx}{\mathbf{x}}
\newcommand{\vy}{\mathbf{y}}

\DeclareMathOperator{\re}{Re}
\DeclareMathOperator{\im}{Im}

\DeclareMathOperator{\sign}{sign}
\DeclareMathOperator*{\meanlim}{l.\! i.\! m.}

\newcommand{\nn}{\nonumber}
\newcommand{\be}{\begin{equation}}
\newcommand{\ee}{\end{equation}}

\newcommand{\ti}{\tilde}
\newcommand{\abs}[1]{\left\lvert #1 \right\rvert}
\newcommand{\abss}[1]{\lvert #1 \rvert}
\newcommand{\norm}[1]{\left\lVert #1 \right\rVert}

\newcommand{\inner}[2]{\left\langle#1,#2\right\rangle}

\newcommand{\floor}[1]{\lfloor #1 \rfloor}

\newcommand{\vphi}{\varphi}

\newcommand{\lam}{\lambda}
\newcommand{\gam}{\gamma}

\numberwithin{equation}{section}

\makeatletter
\newcommand{\dlmf}[1]{%
\cite[%
  \def\nextitem{\def\nextitem{, }}%
  \@for \el:=#1\do{\nextitem\href{http://dlmf.nist.gov/\el}{(\el)}}%
]{dlmf}%
}
\makeatother

\begin{document}

\title{Dispersion Estimates for Spherical Schr\"odinger Equations}

\author[A. Kostenko]{Aleksey Kostenko}
\address{Faculty of Mathematics\\ University of Vienna\\
Oskar-Morgenstern-Platz 1\\ 1090 Wien\\ Austria}
\email{\mailto{duzer80@gmail.com};\mailto{Oleksiy.Kostenko@univie.ac.at}}
\urladdr{\url{http://www.mat.univie.ac.at/~kostenko/}}

\author[G. Teschl]{Gerald Teschl}
\address{Faculty of Mathematics\\ University of Vienna\\
Oskar-Morgenstern-Platz 1\\ 1090 Wien\\ Austria\\ and International 
Erwin Schr\"odinger Institute for Mathematical Physics\\ 
Boltzmanngasse 9\\ 1090 Wien\\ Austria}
\email{\mailto{Gerald.Teschl@univie.ac.at}}
\urladdr{\url{http://www.mat.univie.ac.at/~gerald/}}

\author[J. H. Toloza]{Julio H. Toloza}
\address{Consejo Nacional de Investigaciones Cient\'{i}ficas y 
T\'{e}cnicas (CONICET)\\ and Centro de Investigaci\'{o}n en 
Inform\'{a}tica para la Ingenier\'{i}a\\ Universidad Tecnol\'{o}gica 
Nacional -- Facultad Regional C\'{o}rdoba\\ Maestro L\'{o}pez s/n\\
X5016ZAA C\'{o}rdoba, Argentina}
\email{\mailto{jtoloza@scdt.frc.utn.edu.ar}}

\thanks{Ann. Henri Poincar\'e {\bf 17}, 3147--3176 (2016).}
\thanks{{\it Research supported by the Austrian Science Fund (FWF) 
under Grants No.\ Y330, P26060 and by the National Scientific and 
Technical Research Council (CONICET, Argentina).}}


\begin{abstract}
We derive a dispersion estimate for one-dimensional perturbed radial Schr\"o\-din\-ger operators.
We also derive several new estimates for solutions of the underlying differential equation and
investigate the behavior of the Jost function near the edge of the continuous spectrum.
\end{abstract}

\maketitle

%%%%%%%%%%%%%%%%%%%%%%%%%%%%%%%%%%%%%%%%%%%%%%%%%%%%%%%%%%%%%%%%%%%%%%%%%%%%%%%
\section{Introduction}
%%%%%%%%%%%%%%%%%%%%%%%%%%%%%%%%%%%%%%%%%%%%%%%%%%%%%%%%%%%%%%%%%%%%%%%%%%%%%%%

We are concerned with the one-dimensional Schr\"odinger equation
\begin{equation} \label{Schr}
  \I \dot \psi(t,x) = H \psi(t,x), \quad 
  H:= - \frac{d^2}{dx^2} + \frac{l(l+1)}{x^2} + q(x),\quad 
  (t,x)\in\R\times \R_+,
\end{equation}
with real integrable potential $q$ and with the angular momentum 
$l>-\frac{1}{2}$. We will use $\tau$ to describe the formal Sturm--Liouville differential expression and
$H$ the self-adjoint operator acting in $L^2(\R_+)$ and given by $\tau$ together with the usual boundary condition at $x=0$:
\be\label{eq:bc}
\lim_{x\to0} x^l ( (l+1)f(x) - x f'(x))=0, \qquad l\in\Big(-\frac{1}{2},\frac{1}{2}\Big).
\ee

More specifically, our goal is to provide dispersive 
decay estimates for these equations. To this end we
recall (e.g., \cite[Sect.~9.7]{tschroe}) that for $\int_0^\infty x |q(x)| dx<\infty$ 
the operator $H$ has a purely absolutely continuous spectrum
on $(0,\infty)$ plus a finite number of eigenvalues in $(-\infty,0]$. 
At the edge of the continuous spectrum there could be a resonance (or an eigenvalue if $l>\frac{1}{2}$). Various equivalent definitions of
what is meant by a resonance in this setting will be given in Lemma~\ref{lem:resonance}.
Then our main result read as follows:
%%%%%%%%%%%%%%%%%%%%%%%%%%%%%%%%%%%%%%%%%%%%%%%%%%%%%%%%%%%%%%%%%%%%%%%%%%%%%%%
\begin{theorem}\label{Main}
Assume that 
\be\label{eq:q_hyp}
\int_0^1 |q(x)| dx <\infty \quad \text{and} \quad  
\int_1^\infty x^{\max(2,l+1)} |q(x)| dx <\infty,
\ee
and suppose there is neither a resonance nor an eigenvalue at $0$.
Then the following decay holds
\begin{equation}\label{full}
\norm{\E^{-\I tH}P_c(H)}_{L^1(\R_+)\to L^\infty(\R_+)}
	= \mathcal{O}(|t|^{-1/2}),
\quad t\to\infty.
\end{equation}
Here $P_c(H)$ is the orthogonal projection in $L^2(\R_+)$ onto 
the continuous spectrum of $H$.
\end{theorem}
%%%%%%%%%%%%%%%%%%%%%%%%%%%%%%%%%%%%%%%%%%%%%%%%%%%%%%%%%%%%%%%%%%%%%%%%%%%%%%%
This result will follow from the corresponding low energy result Theorem~\ref{thm:le2} (see also Theorem~\ref{thm:le1}) with the high energy result
Theorem~\ref{thm:he}. We also remark that the decay rate is optimal (see below).

On the whole line such results have a long tradition and we refer to Weder \cite{wed}, Goldberg and Schlag \cite{GS}, Egorova, Kopylova, Marchenko and Teschl \cite{EKMT} (for the discrete case see \cite{EKT}) as well as the reviews \cite{K10,schlag}. On the half line the case $l=0$ was treated by Weder \cite{wed2}. The case of general $l$ but without potential was recently considered in Kova\v{r}\'{i}k and Truc \cite{kotr} (see also \cite{FFFP, FFFP2} for related results).
While our overall strategy looks quite similar to the classical case $l=0$, the details are much more delicate at several points: the first problem stems from
the fact that only one solution will be bounded near $x=0$ while the other one will have a singularity if $l> 0$. In particular, in this case the Jost solutions will
have a singularity near $x=0$ and the expression of the regular solution (which is in the domain of our operator near $x=0$) in terms of the Jost
solutions (i.e., the scattering relations) can no longer be used to obtain useful estimates. The second problem is that the simple group structure
of the exponential functions breaks down for Bessel functions which requires novel strategies to handle the Born series expansion of the resolvent.
And of course one has to work much harder to get some estimates, which are trivial for trigonometric functions, for Bessel functions. In particular,
our present paper should also be understood as a contribution to understanding the properties of solutions of the underlying spectral problem.
In this respect we would like to emphasize that the behavior of the Jost function near the bottom of the essential spectrum is still not understood
satisfactorily, and for this very reason the resonant case had to be excluded from our main theorem. This is definitely a gap which should be filled.

As already mentioned, we have restricted ourselves to the boundary condition \eqref{eq:bc}
corresponding to the Friedrichs extension for $-\frac{1}{2}<l<\frac{1}{2}$. 
%Notwithstanding its relevance from a physical point of view, this choice is related to 
%the present work relying heavily upon results from \cite{kst,kst2,kt,kt2}. 
We will investigate the effect of other boundary conditions (including the case of \eqref{Schr}--\eqref{eq:bc} with $l\in (-\frac32,-\frac12)$, considered in \cite{bdg}) in a forthcoming work \cite{HKT}. For the remaining missing case $l=-\frac{1}{2}$ on the other hand we do expect Theorem~\ref{Main} hold true but,
due to the logarithmic part of the second solution of the Bessel equation, proofs would be significantly more involved so the treatment of this
case has been omitted.

Finally, we mention that one of the motivation to study \eqref{Schr} is the fact that it arises naturally when discussing the $n$-dimensional
Schr\"odinger equation
\begin{equation} \label{Schr-n}
  \I \dot \Psi(t,\vx) = H_n \Psi(t,\vx), \quad 
  H_n:= - \Delta +  V(\vx),\quad 
  (t,\vx)\in\R\times \R^n, \quad n\ge 2.
\end{equation}
However, it is important to emphasize that this is not the only motivation since operators of the type in \eqref{Schr} are the 
prototypical example of strongly singular Schr\"odinger operators which have attracted considerable interest
recently (see e.g.\ \cite{kst,kst2,kt,kt2} and the references therein) or as examples in other physical and mathematical models (see e.g.\ \cite{bdg, DR}).
Nevertheless, and since a lot is known about dispersive estimates for \eqref{Schr-n} (see the reviews \cite{K10,schlag} already
mentioned above), it seems worth while to discuss what these estimates imply for \eqref{Schr}.

To this end recall (see e.g.\ Example 1.5 in \cite{wdln}) that if $V(\vx)=q(x)$, $x=|\vx|$, is radially symmetric, then $H_n$ will be reduced by
the spherical harmonics (cf.\ \cite{muln})
\[
Y_l^m: \bS^{n-1} \to \C, \qquad l\in\N_0, \: m=1,\dots,N(n,l),
\]
which are an orthonormal basis of eigenfunctions of the Laplace--Beltrami operator $\Delta_{\bS^{n-1}}$,
\[
-\Delta_{\bS^{n-1}} Y_l^m = l(l+n-2) Y_l^m,
\]
on $L^2(\bS^{n-1})$.\footnote{The $l$ used here is different from
the $l$ in \eqref{Schr} and the rest of the paper unless $n=3$.} Then the subspaces
\[
\hr_{l,m} = \Big\{ \Psi(\vx)=x^{\frac{n-1}{2}} \psi(x) Y_l^m\Big(\frac{\vx}{x}\Big) \big|\ x=|\vx|,\, \psi\in L^2(\R_+) \Big\} \subseteq L^2(\R^n)
\]
span $L^2(\R^n)= \bigoplus_{l,m} \hr_{l,m}$ and give rise to the decomposition
\[
H_n = \bigoplus_{l,m} U_n^{-1} H_{n,l} U_n, \qquad U_n: \hr_{l,m} \to L^2(\R_+),\ \:  x^{\frac{n-1}{2}} \psi(x) Y_l^m\Big(\frac{\vx}{x}\Big) \mapsto \psi(x),
\]
where
\[
H_{n,l} = -\frac{d^2}{dx^2} + \frac{l(l+n-2)+\frac{(n-1)(n-3)}{4}}{x^2} +q(x).
\]
In particular, an estimate of the type
\be\label{eq:dend}
\norm{\E^{-\I t H_n}P_c(H_n)}_{L^1(\R^n)\to L^\infty(\R^n)}
	= \mathcal{O}(|t|^{-n/2}),
\ee
implies
\[
\norm{\E^{-\I t H_{n,l}} P_c(H_{n,l})}_{L^1(\R_+;x^{\frac{n-1}{2}}) \to L^\infty(\R_+;x^{-\frac{n-1}{2}})}
= \mathcal{O}(|t|^{-n/2}).
\]
Here $L^2(\R_+;x^\alpha)$, $\alpha\in\R$ denotes the standard $L^2$ space with the weight $x^\alpha$.
In the special case $l=0$ we get  
\[
H_{n,0} = -\frac{d^2}{dx^2} + \frac{s(s-1)}{4x^2} + q(x),\quad s= \frac{n-1}{2}, 
\]
and hence 
\[
\norm{\E^{-\I t H_{n,0}} P_c(H_{n,0})}_{L^1(\R_+;x^s) \to L^\infty(\R_+;x^{-s})}
= \mathcal{O}(|t|^{-s-1/2}),
\]
which generalizes Theorem~2.4 from \cite{kotr} where the case without potential and with
the weight $(1+x)^s$ was established. For conditions on $V$ for \eqref{eq:dend} to
hold we refer again to the above mentioned survey articles \cite{K10,schlag}. At this point we only
note that it of course holds in the case without potential where the time evolution is given by
\[
\left(\E^{\I \Delta t} \Psi_0\right)(\vx) = \frac{1}{ (4\pi\I t)^{n/2}} \int_{\R^n} \E^{\I\frac{|\vx-\vy|^2}{4t}} \Psi_0(\vy) d\vy.
\]
Moreover, the time evolution of $H_{n,l}$ can be obtained by projecting $\E^{\I \Delta t}$ to the corresponding
spherical harmonics. For example, in three dimensions one obtains

\begin{align*}
[\E^{-\I t H_{3,l}}](x,y)
	&= \frac{xy\,Y_l^0(0,0)^{-1}}{(4 \pi \I t)^{3/2}} 
	   \int_0^\pi \int_0^{2\pi} \E^{\I\frac{x^2-2xy\cos(\theta) 
	   + y^2}{4t}} Y_l^0(\theta,\vphi) \sin(\theta) d\theta\,d\vphi
	   \\
%& \quad = \frac{xy}{(4 \pi \I t)^{3/2}} \int_0^\pi \int_0^{2\pi} \E^{\I\frac{x^2-2xy\cos(\theta) +y^2}{4t}} P_l(\cos(\theta)) \sin(\theta) d\theta\,d\vphi\\
	&= \frac{2\pi xy}{(4 \pi \I t)^{3/2}} \E^{\I\frac{x^2 +y^2}{4t}} 
	   \int_{-1}^1 \E^{-\I\frac{xy}{2t} r} P_l(r) dr
	   \\
	&= \frac{\I^{-l-1/2} }{2 \I t} \E^{\I\frac{x^2 +y^2}{4t}} \sqrt{xy} 
	   J_{l+1/2}\Big(\frac{xy}{2t}\Big),
\end{align*}
where we have chosen $m=0$, $\vx=(0,0,x)$ and used 
\[
Y_l^0(\theta,\vphi) = \sqrt{\frac{2l+1}{4\pi}} P_l(\cos(\theta)), \quad P_l (\xi) = \frac{1}{2^l l!}\frac{d^l}{d\xi^l} (\xi^2-1)^l,
\]
 as well as
\dlmf{18.17.19} for the last integral. Here $J_\nu$ is the Bessel function of order $\nu$ and $P_l$ are the Legendre polynomials.
This should again be compared with \cite[Eq.~(3.23)]{kt}.
In particular,  for $l=0$ we have $\| [\E^{-\I t H_{3,0}}](x,y) \|_\infty = \frac{1}{\sqrt{\pi} |t|^{1/2}}$ (while $\| (xy)^{-1} [\E^{-\I t H_{3,0}}](x,y)\|_\infty = \frac{1}{2\sqrt{\pi} |t|^{3/2}}$), which shows that the decay in our main Theorem~\ref{Main} is optimal.
Of course a potential of the type $V(\vx)=\frac{a}{|\vx|^2}$ with $a> - \frac{(n-2)^2}{4}$ could be included in this discussion as it can be absorbed in the definition of $s$ \cite{bpstz1,bpstz2}.

%%%%%%%%%%%%%%%%%%%%%%%%%%%%%%%%%%%%%%%%%%%%%%%%%%%%%%%%%%%%%%%%%%%%%%%%%
\section{Properties of solutions}
\label{wa-sec}
%%%%%%%%%%%%%%%%%%%%%%%%%%%%%%%%%%%%%%%%%%%%%%%%%%%%%%%%%%%%%%%%%%%%%%%%%

In this section we will collect some properties of the solutions of the underlying
differential equation required for our main results.

\subsection{The regular solution}

Suppose that $l>-\frac{1}{2}$ and  
\be\label{q:hyp}
q\in L^1_\loc(\R_+)\quad \text{and} \quad \int_0^1 x|q(x)|dx <\infty.
\ee
Then the ordinary differential equation 
\[
\tau f= z f,\qquad \tau := - \frac{d^2}{dx^2} + \frac{l(l+1)}{x^2} + q(x),
\]
has a system of solutions $\phi(z,x)$ and $\theta(z,x)$ which is real 
entire with respect to $z$ such that 
\be\label{eq:fs01}
\phi(z,x) 
	= C_l x^{l+1}\ti{\phi}(z,x),\quad  
\theta(z,x)
	= \frac{x^{-l}}{(2l+1)C_l}\ti{\theta}(z,x),\quad 
C_l =\frac{\sqrt{\pi}}{\Gamma(l+\frac{3}{2}) 2^{l+1}},
\ee
where $\ti{\phi}(z,\cdot),\ti{\theta}(z,\cdot)\in W^{1,1}[0,1]$ and 
$\ti{\phi}(z,0)=\ti{\theta}(z,0)=1$. For a detailed construction of these 
solutions we refer to, e.g., \cite{kt}.
 
We start with two lemmas containing estimates for the Green's function 
of the unperturbed equation 
\[ 
G_l(z,x,y) = \phi_l(z,x) \theta_l(z,y) - \phi_l(z,y) \theta_l(z,x)
\]
and the regular solution $\phi(z,x)$ 
(see, e.g., \cite[Lemmas~2.2, A.1, and A.2]{kst}).  Here 
\[
\phi_l(z,x) 
	= z^{-\frac{2l+1}{4}} \sqrt{\frac{\pi x}{2}} 
	  J_{l+\frac{1}{2}}(\sqrt{z} x),
\]
and
\[
\theta_l(z,x) 
= z^{\frac{2l+1}{4}} \sqrt{\frac{\pi x}{2}} 
	\begin{dcases}
	\frac{1}{\sin((l\!+\!\frac{1}{2})\pi)} 
		J_{-l-\frac{1}{2}}(\sqrt{z} x), 
		& {l\!+\!\frac{1}{2}} \in \R_+\!\setminus\N_0,
		\\
		\frac{1}{\pi}\log(z)J_{l+\frac{1}{2}}(\sqrt{z} x)
	    -Y_{l+\frac{1}{2}}(\sqrt{z} x), & {l\!+\!\frac{1}{2}} \in\N_0,
	\end{dcases}
\]
where $J_{\nu}$ and $Y_{\nu}$ are the usual Bessel and Neumann 
functions (see Appendix \ref{sec:Bessel}). All branch
cuts are chosen along the negative real axis unless explicitly stated
otherwise.

\begin{lemma}[\cite{kst}] \label{lem:b.1}
For $l>-\frac{1}{2}$ the following estimates hold:
\be\label{estphil}
\abs{\phi_l(k^2,x)} 
	\le C \left(\frac{x}{1+\abs{k}x}\right)^{l+1} \E^{\abs{\im\, k} x},
\ee	
and
\be\label{estGl}
\abs{G_l(k^2,x,y)} 
	\le C \left(\frac{x}{1+\abs{k}x}\right)^{l+1} 
	\left(\frac{1+\abs{k}y}{y}\right)^l 
	\E^{\abs{\im\, k} (x-y)}, \quad y \leq x.
\ee
\end{lemma}

\begin{lemma}[\cite{kst}]
Assume \eqref{q:hyp}. Then $\phi(z,x)$ satisfies the integral equation
\[ 
\phi(z,x) = \phi_l(z,x) + \int_0^x G_l(z,x,y) q(y) \phi(z,y) dy.
\]
Moreover, $\phi$ is entire in $z$ for every $x>0$ and satisfies the estimate
\be\label{estphi}
\abs{\phi(k^2,x) - \phi_l(k^2,x)} 
	\leq C \left(\frac{x}{1+ \abs{k} x}\right)^{l+1} 
	\E^{\abs{\im\, k} x} \int_0^x \frac{y \abs{q(y)}}{1 +\abs{k} y} dy.
\ee
\end{lemma}

We also need the following estimates.

\begin{lemma}
\label{lem:part-z}
For $l>  -\frac{1}{2}$ the following estimates
hold
\be\label{eq:partial-z-phil}
\abss{\partial_k{\phi}_l(k^2,x)}
\le C |k|x 
	\left(\frac{ x}{1+\abs{k}x}\right)^{l+2}
	\E^{\abss{\im k} x}
\ee
and 
\be\label{eq:partial-z-Gl}
\abs{\partial_k G_l(k^2,x,y)}
\le C|k| x \left(\frac{x}{1+|k|x}\right)^{l+2}
	\left(\frac{1+|k|y}{y}\right)^{l}
	\E^{\abss{\im k}(x-y)},\quad y\le x.
\ee
\end{lemma}
\begin{proof}
The first inequality follows from the identity (see \eqref{eq:recurrence})
\[
\partial_k{\phi}_l(k^2,x) = -kx \,\phi_{l+1}(k^2,x)
\]
along with the bound \eqref{estphil}.

Before proving \eqref{eq:partial-z-Gl}, let us mention that 
\begin{align*}
G_l(k^2,x,y) 
&= - \frac{\pi}{2}\sqrt{xy} \left[J_{l+\frac12}(kx)Y_{l+\frac12}(ky) 
   - J_{l+\frac12}(ky)Y_{l+\frac12}(kx)\right],
\\
&= - \frac{\I\pi}{4}\sqrt{xy} \left[ H^{(1)}_{l+\frac12}(kx)H^{(2)}_{l+\frac12}(ky)    
   - H^{(1)}_{l+\frac12}(ky)H^{(2)}_{l+\frac12}(kx)\right],
\end{align*}
where $H^{(1)}_{\nu}$ and $H^{(2)}_{\nu}$ are the usual Hankel functions (see Appendix \ref{sec:Bessel}).
Hence we obtain 
\begin{align}\label{eq:partial-z-Gl-proof}
\begin{split}
\partial_k G_l(k^2,x,y) 
=&\; \frac{\pi}{2}\sqrt{xy} \left[xJ_{l+\frac32}(kx)Y_{l+\frac12}(ky) 
     - y J_{l+\frac32}(ky)Y_{l+\frac12}(kx)\right]
\\
 &\; - \frac{\pi}{2}\sqrt{xy} \left[y J_{l+\frac12}(kx)Y_{l-\frac12}(ky) 
     - x J_{l+\frac12}(ky)Y_{l-\frac12}(kx)\right],
\\
=&\; \frac{\I\pi}{4}\sqrt{xy} \left[xH^{(1)}_{l+\frac32}(kx)H^{(2)}_{l+\frac12}(ky)
     - y H^{(1)}_{l+\frac32}(ky)H^{(2)}_{l+\frac12}(kx)\right]
\\
 &\;-\frac{\I\pi}{4}\sqrt{xy}\left[y H^{(1)}_{l+\frac12}(kx)H^{(2)}_{l-\frac12}(ky) 
    - x H^{(1)}_{l+\frac12}(ky)H^{(2)}_{l-\frac12}(kx)\right].
\end{split}
\end{align}
Consider the function 
\begin{align*}
\begin{split}
{\mathcal G}_{l}(\eta,\xi) 
:=&\; \frac{\pi}{2} \big[\eta J_{l+\frac32}(\eta)Y_{l+\frac12}(\xi)
	- \xi J_{l+\frac32}(\xi)Y_{l+\frac12}(\eta) 
\\
  &   \qquad\qquad\qquad\qquad
	-\;\xi J_{l+\frac12}(\eta)Y_{l-\frac12}(\xi)
	+  \eta J_{l+\frac12}(\xi)Y_{l-\frac12}(\eta) \big]
\\
 =&\; \frac{\I\pi}{4} \left[ \eta H^{(1)}_{l+\frac32}(\eta)H^{(2)}_{l+\frac12}(\xi) 
    - \xi H^{(1)}_{l+\frac32}(\xi)H^{(2)}_{l+\frac12}(\eta)\right.
\\    
  &   \qquad\qquad\qquad\qquad
      \left. 
    -\;\xi H^{(1)}_{l+\frac12}(\eta)H^{(2)}_{l-\frac12}(\xi) 
    +  \eta H^{(1)}_{l+\frac12}(\xi)H^{(2)}_{l-\frac12}(\eta)\right].
\end{split}
\end{align*}

{\em Step (i):   $\abs{\xi}\le\abs{\eta}\le 1$.} Let us estimate the function
\[
\eta J_{l+\frac32}(\eta)Y_{l+\frac12}(\xi)
	-\xi J_{l+\frac32}(\xi)Y_{l+\frac12}(\eta).
\]
Employing \dlmf{10.2.2} and the monotonicity of $x\mapsto \frac{x}{1+x}$ on $\R_+$, we get
\begin{multline*}
\abs{\eta J_{l+\frac32}(\eta)Y_{l+\frac12}(\xi)
	-\xi J_{l+\frac32}(\xi)Y_{l+\frac12}(\eta)}\\
\begin{aligned}	
&\le C \left[|\eta| \left(\frac{\abs{\eta}}{1+\abs{\eta}}\right)^{l+\frac{3}{2}}
	   \left(\frac{1+\abs{\xi}}{\abs{\xi}}\right)^{l+\frac12} 
	   + |\xi| \left(\frac{\abs{\xi}}{1+\abs{\xi}}\right)^{l+\frac32}
	   \left(\frac{1+\abs{\eta}}{\abs{\eta}}\right)^{l+\frac12}\right]
	   \\
&\le C |\eta|\left(\frac{\abs{\eta}}{1+\abs{\eta}}\right)^{l + \frac32}
	   \left(\frac{1+\abs{\xi}}{\abs{\xi}}\right)^{l+\frac12} 
 \le C \left(\frac{\abs{\eta}}{1+\abs{\eta}}\right)^{l + \frac32}
       \left(\frac{1+\abs{\xi}}{\abs{\xi}}\right)^{l+\frac12}.
\end{aligned}			
\end{multline*}
Similarly, if $l>1/2$, then 
\begin{multline*}
\abs{\xi J_{l+\frac12}(\eta)Y_{l-\frac12}(\xi)
	-\eta J_{l+\frac12}(\xi)Y_{l-\frac12}(\eta)}\\
\begin{aligned}	
&\le C\left[|\xi| \left(\frac{\abs{\eta}}{1+\abs{\eta}}\right)^{l+\frac{1}{2}}
			\left(\frac{1+\abs{\xi}}{\abs{\xi}}\right)^{l-\frac12} +
			|\eta| \left(\frac{\abs{\xi}}{1+\abs{\xi}}\right)^{l+\frac12} \left(\frac{1+\abs{\eta}}{\abs{\eta}}\right)^{l-\frac12}\right]\\
&\le C |\eta|  \left(\frac{\abs{\eta}}{1+\abs{\eta}}\right)^{l+\frac12} \left(\frac{1+\abs{\xi}}{\abs{\xi}}\right)^{l-\frac12}
\le C  |\eta|\left(\frac{\abs{\eta}}{1+\abs{\eta}}\right)^{l+\frac32} \left(\frac{1+\abs{\xi}}{\abs{\xi}}\right)^{l+\frac12}.	
\end{aligned}			
\end{multline*}
If $|l| <1/2$, then using \eqref{eq:Jnu01} and \eqref{eq:Ynu01} we obtain
\begin{multline*}
\abs{\xi J_{l+\frac12}(\eta)Y_{l-\frac12}(\xi)
	-\eta J_{l+\frac12}(\xi)Y_{l-\frac12}(\eta)}\\
\begin{aligned}	
&\le C  \left[|\xi| \left(\frac{\abs{\eta}}{1+\abs{\eta}}\right)^{l+\frac{1}{2}}
			\left(\frac{\abs{\xi}}{1+\abs{\xi}}\right)^{\frac12 - l} +
			|\eta| \left(\frac{\abs{\xi}}{1+\abs{\xi}}\right)^{l+\frac12} \left(\frac{\abs{\eta}}{1+\abs{\eta}}\right)^{\frac12-l}\right]\\
&\le C  |\eta| \left(\frac{\abs{\eta}}{1+\abs{\eta}}\right)^{l+\frac32} \left(\frac{1+\abs{\xi}}{\abs{\xi}}\right)^{l+\frac12}. 
\end{aligned}			
\end{multline*}
Finally, for $l=1/2$ we get
\begin{align*}
\abs{\xi J_{1}(\eta)Y_{0}(\xi)
	-\eta J_{1}(\xi)Y_{0}(\eta)}
\le C& \frac{\abs{\eta}}{1+\abs{\eta}} \frac{\abs{\xi}}{1+\abs{\xi}}
	   \log\left(\frac{\abs{\eta}}{\abs{\xi}}\right)
	   \\
\le C& |\eta|\left(\frac{\abs{\eta}}{1+\abs{\eta}}\right)^{2} 
	   \left(\frac{1+\abs{\xi}}{\abs{\xi}}\right).
\end{align*}

Summarizing the above, we find that the function ${\mathcal G}_l$ admits
the following estimate
\[
|{\mathcal G}_l(\eta,\xi)| 
\le C|\eta|\left(\frac{\abs{\eta}}{1+\abs{\eta}}\right)^{l + \frac32}
	 \left(\frac{1+\abs{\xi}}{\abs{\xi}}\right)^{l+\frac12}
\]
if $0< |\xi|\le |\eta|\le 1 $ and $l > -1/2$.

{\em Step (ii):  $\abs{\xi}\le 1\le\abs{\eta}$.} First, we get
\begin{equation*}
\abs{\eta J_{l+\frac32}(\eta)Y_{l+\frac12}(\xi)
	-\xi J_{l+\frac32}(\xi)Y_{l+\frac12}(\eta)} 
\le C\sqrt{ |\eta|}\left(\frac{1+\abs{\xi}}{\abs{\xi}}\right)^{l+\frac12}\E^{\abs{\im\eta}}
\end{equation*}
as implied by \eqref{eq:asymp-J-infty} and \eqref{eq:asymp-Y-infty}. 
If $l>1/2$, we get 
\[
\abs{\xi J_{l+\frac12}(\eta)Y_{l-\frac12}(\xi)
	-\eta J_{l+\frac12}(\xi)Y_{l-\frac12}(\eta)} 
\le C\sqrt{ |\eta|}\left(\frac{1+\abs{\xi}}{\abs{\xi}}\right)^{l+\frac12}\E^{\abs{\im\eta}}\,.
\]
For $|l| <1/2$ we obtain
\[
\abs{\xi J_{l+\frac12}(\eta)Y_{l-\frac12}(\xi)
	-\eta J_{l+\frac12}(\xi)Y_{l-\frac12}(\eta)} 
\le C\sqrt{ |\eta|}\left(\frac{\abs{\xi}}{1+\abs{\xi}}\right)^{l+\frac12 }\E^{\abs{\im\eta}}\,.
\]
And finally, for $l=1/2$, we get
\begin{align*}
\abs{\xi J_{1}(\eta)Y_{0}(\xi)
	-\eta J_{1}(\xi)Y_{0}(\eta)}
\le C &\sqrt{ |\eta|}\frac{\abs{\xi}}{1+\abs{\xi}}
	  \log\left(\frac{1+\abs{\xi}}{\abs{\xi}}\right)\E^{\abs{\im\eta}}
	  \\
\le C& \sqrt{ |\eta|}\frac{1+\abs{\xi}}{\abs{\xi}}\E^{\abs{\im\eta}}\,.
\end{align*}

Summarizing the above, we find that the function ${\mathcal G}_l$ admits
the following estimate
\[
\abs{{\mathcal G}_l(\eta,\xi)} 
\le C\sqrt{\abs{\eta}}
     \left(\frac{1+\abs{\xi}}{\abs{\xi}}\right)^{l+\frac12}
     \E^{\abs{\im\eta}}
\]
if $0<\abs{\xi}\le 1\le \abs{\eta}$ and $l> -1/2$.

{\em Step (iii): $1\le\abs{\xi}\le\abs{\eta}$.} To deal with the remaining case we shall use the second equality in \eqref{eq:partial-z-Gl-proof} and the asymptotic 
expansions of Hankel functions \eqref{eq:asymp-H1-infty}--\eqref{eq:asymp-H2-infty}: 
\be
\abs{{\mathcal G}_l(\eta,\xi)} 
\sim 2\cos(\eta - \xi) \left(\sqrt{\frac{\eta}{\xi}} 
						   - \sqrt{\frac{\xi}{\eta}}\right)
\ee
as $|\eta|$, $|\xi|\to\infty$. Therefore, we get
\[
\abs{{\mathcal G}_l(\eta,\xi)} 
\le  C\sqrt{ \frac{|\eta|}{|\xi|}}\E^{\abs{\im(\eta - \xi)}}
\]
if $1\le  |\xi|\le  |\eta| $ and $l> -1/2$.

Combining all these estimates for the function ${\mathcal G}_l$ with 
the equality
\[
G_l(k^2,x,y) = \frac{\sqrt{xy}}{k} {\mathcal G}_l(kx,ky),
\]
after straightforward calculations we arrive at \eqref{eq:partial-z-Gl}.
\end{proof}

\begin{lemma}
\label{lem:about-partial-z}
Assume \eqref{q:hyp}. Then  
$\partial_k{\phi}(k^2,x)$ is a solution to the integral equation
\begin{multline}
\label{eq:partial-z-phi}
\partial_k {\phi}(k^2,x)
= \partial_k{\phi}_l(k^2,x)\\
 	+ \int_0^x [\partial_k G_l(k^2,x,y)] {\phi}(k^2,y) 
 	+ G_l(k^2,x,y) \partial_k{\phi}(k^2,y)]q(y)dy 
\end{multline}
and satisfies the estimate
\begin{equation}
\label{eq:partial-z-diff-phi}
\abss{\partial_k{\phi}(k^2,x) -\partial_k {\phi}_l(k^2,x)}
\le C|k|x \left(\frac{x}{1+ \abs{k} x}\right)^{l+2} \E^{\abs{\im\, k} x} 
	\int_0^x \frac{y\abs{q(y)}}{1 +\abs{k} y} dy.
\end{equation}
\end{lemma}
\begin{proof}
The proof is based on the successive iteration procedure (see, e.g., \cite[Chap.~I.5]{cs}). As in the proof of Lemma 2.2 in \cite{kst}, set
\[
\phi = \sum_{n=0}^\infty \phi_{n},\quad \phi_0:=\phi_l,\quad \phi_n(k^2,x):=\int_0^x G_l(k^2,x,y)\phi_{n-1}(k^2,y)q(y)dy
\]
for all $n\in\N$. The series is absolutely convergent since 
\be\label{eq:phi_n}
\abs{{\phi}_n(k^2,x)}
\le \frac{C^{n+1}}{n!}
	\left(\frac{ x}{1+|k|x}\right)^{l+1}
	\E^{\abss{\im k} x}
	\left(\int_0^x \frac{y\abs{q(y)}}{1+|k|y}dy\right)^n.
\ee

Similarly, let us show that $\partial_k \phi(k^2,x)$ given by
\begin{gather}
\partial_k \phi = \sum_{n=0}^\infty \beta_n,\quad  \beta_0(k,x)
	:=\partial_k{\phi}_l(k^2,x),\label{eq:betaser}
	\\
\begin{multlined}	
\beta_{n}(k,x)
	:=\int_0^x \partial_k G_l(k^2,x,y)\,  {\phi}_{n-1}(k^2,y) q(y)dy
		\qquad\qquad\qquad\qquad
		\\
		+ \int_0^x G_l(k^2,x,y) \beta_{n-1}(k,y) q(y)dy,
		\quad n\in \N,\label{eq:partial-z-first}
\end{multlined}		 
\end{gather}
satisfies \eqref{eq:partial-z-phi}. Using \eqref{eq:phi_n} and 
\eqref{eq:partial-z-phil}, we can bound the first summand in 
\eqref{eq:partial-z-first} as follows
\begin{align*}
\abs{\text{1st term}}
&\le
\frac{C^{n+1}}{(n-1)!} |k|x
	\left(\frac{x}{1+|k|x}\right)^{l+2}\! \E^{\abss{\im k} x}
	\int_0^x \frac{y\abs{q(y)}}{1+|k|y}\left(\int_0^y 
	\frac{t\abs{q(t)}}{1+|k|t}dt\right)^{n-1}\!dy
	\\
&\le
\frac{C^{n+1}}{n!}|k| x \left(\frac{x}{1+|k|x}\right)^{l+2} \E^{\abss{\im k} x}
	\left(\int_0^x  \frac{y |q(y)|}{1+|k|y}dy\right)^{n}.
\end{align*}
Next, using induction, one can show that the second summand admits a 
similar bound and hence we finally get
\[
\abs{\beta_n(k,x)}
\le \frac{C^{n+1}}{n!} |k|x \left(\frac{x}{1+|k|x}\right)^{l+2} \E^{\abss{\im k} x}
	\left(\int_0^x  \frac{y |q(y)|}{1+|k|y}dy\right)^{n}.
\] 
This immediately implies the convergence of 
\eqref{eq:betaser} and, moreover, the estimate
\[
\abss{\partial_k{\phi}(k^2,x) -\partial_k{\phi}_l(k^2,x)}
\le \sum_{n=1}^\infty\abs{\beta_n(k,x)},
\]
from which \eqref{eq:partial-z-diff-phi} follows under the assumption
\eqref{q:hyp}.
\end{proof}

Furthermore, by \cite{fad, cc}, the regular solution $\phi$ admits a representation by means of transformation operators preserving the behavior of solutions at $x=0$ (see also \cite[Chap. III]{cs} for further details and historical remarks). 

\begin{lemma}[\cite{cc}]\label{lem:toGL}
Suppose  $q\in L^1_{\loc}([0,\infty))$. Then 
\be\label{eq:to_GL}
\phi(z,x) = \phi_l(z,x) + \int_0^x B(x,y) \phi_l(z,y) dy =: (I+B)\phi_l(z,x),
\ee
where the so-called Gelfand--Levitan kernel $B:\R_+^2\to \R$ satisfies the estimate
\be\label{eq:GLest}
|B(x,y)| \le \frac{1}{2} {\sigma}_0\left(\frac{x+y}{2}\right)\E^{{\sigma}_1(x)},\quad {\sigma}_j(x):=\int_0^x y^j |q(y)|dy,
\ee
for all $0<y<x$ and $j\in \{0,1\}$.
\end{lemma}

In particular, this lemma immediately implies the following useful result.

\begin{corollary}\label{cor:Best}
Suppose $q\in L^1((0,1))$. Then $B$ is a bounded operator on $L^\infty((0,1))$.
\end{corollary}

\begin{proof}
If $f\in L^\infty(\R_+)$, then using the estimate \eqref{eq:GLest} we get
\begin{multline*}
|(B f)(x)| 
= \Big|\int_0^x B(x,y) f(y) dy\Big| \le \|f\|_\infty \int_0^x |B(x,y)|dy 
\\
\le \frac{1}{2}\|f\|_\infty \E^{{\sigma}_1(1)} \int_0^x {\sigma}_0\Big(\frac{x+y}{2}\Big)dy 
\le \frac{1}{2}  \|f\|_\infty \E^{{\sigma}_1(1)} \sigma_0(1),
\end{multline*}
which proves the claim.
\end{proof}

\begin{remark}\label{rem:Best}
Note that $B$ is a bounded operator on $L^2((0,a))$ for all $a>0$. However, the estimate \eqref{eq:GLest} allows to show that its
norm behaves like $\OO(a)$ as $a\to \infty$ and hence $B$ might not be bounded on $L^2(\R_+)$. 
\end{remark}

\subsection{The singular Weyl function}

{\em The singular Weyl function} $m:\C\setminus\R\to \C$ is defined 
such that 
\be\label{eq:s_m}
\psi(z,x)= \theta(z,x) + m(z) \phi(z,x),\quad z\in\C\setminus\R, 
\ee
belongs to $L^2((1,\infty))$. Note that, while the first solution 
$\phi(z,x)$ is unique under the normalization \eqref{eq:fs01}, the 
second solution $\theta(z,x)$ is not, since for any real entire function $E$ 
the new solution $\ti{\theta}(z,x)=\theta(z,x)-E(z)\phi(z,x)$ also 
satisfies \eqref{eq:fs01}. Note that the corresponding singular 
$m$-function $\ti{m}$ is given by
\[
\ti{m}(z)=m(z)+E(z)
\]
in this case. Moreover, it was shown in \cite{kst2,kt} that 
the singular $m$-function \eqref{eq:s_m} admits the following integral 
representation 
\be\label{eq:int_rep}
m(z) = \ti{E}(z)+(1+z^2)^{\kappa_l}
	   \int_{\R}\Big(\frac{1}{\lam-z}-\frac{\lam}{1+\lam^2}\Big)
	   \frac{d\rho(\lam)}{(1+\lam^2)^{\kappa_l}},\quad z\notin\R.
\ee
Here $\kappa_l:=\floor{\frac{l}{2}+\frac{3}{4}}$ (with $\floor{.}$ the usual floor function), the function $\ti{E}$ is 
real entire, and $\rho:\R\to\R$ is a nondecreasing function satisfying
\[
\rho(\lam)=\frac{\rho(\lam+)+\rho(\lam-)}{2},\quad \rho(0)=0,\quad 
\int_{\R}\frac{d\rho(\lam)}{(1+\lam^2)^{\kappa_l+1}}<\infty.
\]
The operator $H$ is unitarily equivalent to multiplication by the 
independent variable in $L^2(\R,d\rho)$ and thus $\rho$ is called 
{\em the spectral function} and $d\rho$ is {\em the spectral measure}.
Indeed, one has $\mathcal{F}:L^2(\R_+)\to L^2(\R,d\rho)$ defined via
\[
\varphi(x)\mapsto\hat{\varphi}(\lambda)
:=\meanlim_{c\to\infty}\int_0^c\phi(\lambda,x)\varphi(x)dx,
\]
 and its inverse mapping 
$\mathcal{F}^{-1}:L^2(\R,d\rho)\to L^2(\R_+)$ given by
\[
\hat{\varphi}(\lambda)\mapsto\varphi(x)
	:=\meanlim_{r\to\infty}\int_{-r}^r\phi(\lambda,x)
      \hat{\varphi}(\lambda)\rho(d\lambda).
\] 
Here ``$\meanlim$'' denotes the limit in the corresponding $L^2$-norm.
Then, for any Borel function $f$, one has $\mathcal{F}f(H)\mathcal{F}^{-1}$ equal to multiplication 
by $f(\lambda)$.

We also remark that the value of $\kappa_l$ in \eqref{eq:int_rep} is 
the best possible one as the following extension of Marchenko's asymptotic 
formula  shows.

\begin{theorem}[\cite{kt2}]\label{th:main}
Suppose that $q$ satisfies \eqref{q:hyp} and  $m$ is the singular 
$m$-function \eqref{eq:s_m}. Then there is a real entire function $E$ 
such that in any nonreal sector,
\[
m(z) - E(z)=m_l(z)(1+o(1)),\quad |z|\to+\infty,
\]
where
\be\label{eq:II.11}
m_l(z) 
= \begin{dcases}
	\frac{-1}{\sin((l+\frac{1}{2})\pi)} (-z)^{l+\frac{1}{2}}, 
	& {l+\frac{1}{2}}\in\R_+\!\setminus \N_0,
	\\
	\frac{-1}{\pi} z^{l+\frac{1}{2}}\log(-z), 
	& {l+\frac{1}{2}} \in\N_0.
	\end{dcases}
\ee
Moreover, the spectral function satisfies
\be\label{eq:rhoMar}
\rho(\lam)=\rho_l(\lam)(1+o(1)),\quad \lam\to +\infty,
\ee
where
\[
\rho_l(\lam) = \frac{1}{\pi(l+\frac{3}{2})}\id_{[0,\infty)}(\lam) 
			   \lam^{l+\frac{3}{2}},  \qquad l \geq -\frac{1}{2}.
\]
\end{theorem}   

Note that the formula \eqref{eq:rhoMar} was first announced in \cite{ka56}. For extensions of
Theorem~\ref{th:main} to the case when $q$ is a distribution in $H^{-1}_{\loc}$ we refer to \cite{ekt}.

\subsection{The Jost solution}
\label{sec:jsol}

In this subsection, we assume that the potential $q$ belongs to 
{\em the Marchenko class}, i.e., in addition to \eqref{q:hyp}, $q$ 
also satisfies 
\be\label{eq:q_mar}
\int_1^\infty x |q(x)|dx<\infty.
\ee
Recall that under these assumptions on $q$ the spectrum of $H$ is 
purely absolutely continuous on $(0,\infty)$ with an at most finite 
number of eigenvalues $\lam_n \in (-\infty,0]$.

Next we need some estimates for the Weyl solution $\psi$ defined by \eqref{eq:s_m}. We begin 
with some basic properties of the unperturbed Bessel equation in which
case the Weyl solution is given by
\[
\psi_l(k^2,x)	
= \I  k^{l+\frac{1}{2}}\sqrt{\frac{\pi x}{2}}
	H^{(1)}_{l+\frac{1}{2}}(k x),
\]
which is analytic in $\im\, k>0$ and continuous in $\im\, k\ge 0$. Here $H^{(1)}_{\nu}$ is the Hankel function
of the first kind (see Appendix \ref{sec:Bessel}). Its derivative is given by (cf. \eqref{eq:recurrence})
\[
\partial_k \psi_l(k^2,x) 
	= \I k^{l+\frac{1}{2}} x\sqrt{\frac{\pi x}{2}}
	  H^{(1)}_{l-\frac{1}{2}}(k x).
\]
The analog of Lemma~\ref{lem:b.1} reads:

\begin{lemma}\label{lem:b.3}

If $l>-1/2$, then for every $x>0$
\be
\label{est:psi_lA}
\abs{\psi_l(k^2,x)}
	\le C\left(\frac{x}{1+ \abs{k} x}\right)^{-l} \E^{-\abs{\im\, k} x}
\ee
and
\be 
\label{eq:partial-z-weyl-sol-free}
\abs{\partial_k \psi_l(k^2,x)}
	\le C \E^{-\abs{\im k} x} 
	\begin{dcases} 
	|k|x   \left(\frac{1+ \abs{k} x}{x}\right)^{l-1}, & l\ge \frac12,
	\\
	|k|^l x\left(\frac{\abs{k}x}{1+\abs{k} x}\right)^{l}, 
		   &\abs{l}<\frac{1}{2}.
	\end{dcases}
\ee
\end{lemma}

A solution $f(k,\cdot)$ to $\tau y = k^2y$ satisfying the following asymptotic normalization 
\be\label{eq:JostSol}
f(k,x) = \E^{\I k x}(1 + o(1)),\qquad f'(k,x) = \I k \E^{\I k x}(1 + o(1))
\ee
as $x\to \infty$,  is called {\em the Jost solution}. 
In the case $q=0$ we have (cf. \eqref{eq:asymp-H1-infty})
\[
f_l(k,x)
	= \frac{\E^{\I\frac{\pi l}{2}}}{k^l} \psi_l(k^2,x)
	= \I\E^{\I\frac{\pi l}{2}}\sqrt{\frac{\pi xk}{2}}
	H_{l+\frac12}^{(1)}(kx).
\]
\begin{lemma}\label{lem:b.4}
Assume \eqref{eq:q_mar}. Then $f(k,x)$ satisfies the integral equation
\[
f(k,x) 
	= f_l(k,x) - \int_x^\infty G_l(k^2,x,y) q(y) f(k,y) dy.
\]
If $l>-1/2$, then for all $x>0$, $f(\cdot,x)$ is analytic in the upper half plane and can 
be continuously extended to the real axis away from $k=0$ and
\be
\label{estpsi}
\abss{f(k,x) - f_l(k,x)}
	\leq C \left(\frac{\abs{k} x}{1+ \abs{k} x}\right)^{-l} 
	\E^{-\abs{\im\, k}\, x} \int_x^\infty 
	\frac{y q(y)}{1 +\abs{k} y} dy.
\ee
Moreover, the function $h(k,x):= \E^{-\I k x} f(k,x)$ satisfies the estimates
\be\label{est:hlk}
\abs{\partial _k h_l(k,x)} 
	\le \frac{C}{x\abs{k}^2} 
		\left(\frac{1+ \abs{k} x}{\abs{k} x}\right)^{l-1},
\ee
and 
\be\label{est:hk}
\abs{\partial _k h(k,x) -  \partial _k h_l(k,x)} 
	\le \frac{C}{\abs{k}} 
	\left(\frac{1+ \abs{k} x}{\abs{k} x}\right)^{l} 
	\int_x^\infty y|q(y)|dy .
\ee
\end{lemma}

\begin{proof}
The proof is based on the successive iteration procedure. Set
\[
f = \sum_{n=0}^\infty f_{n},\quad 
f_0 := f_l,\quad 
f_n(k,x) := -\int_x^\infty G_l(k^2,x,y)f_{n-1}(k,y)q(y)dy
\]
for all $n\in\N$. The series is absolutely convergent since 
\be\label{eq:f_n}
\abs{{f}_n(k,x)}
\le \frac{C^{n+1}}{n!}
	\left(\frac{1+|k|x}{|k|x}\right)^{l}
	\E^{\abss{\im k} x}
	\left(\int_x^\infty \frac{y\abs{q(y)}}{1+|k|y}dy\right)^n.
\ee
The latter also proves \eqref{estphi}.

The proof of \eqref{est:hlk} is given in Appendix \ref{sec:Bessel}. It remains to prove \eqref{est:hk}. 
First, notice that $h$ solves the following equation
\[ 
h(k,x) = h_l(k,x) - \int_x^\infty \tilde{G}_l(k,x,y) q(y) h(k,y) dy,
\quad 
\ti{G}_l(k,x,y):= G_l(k^2,x,y)\E^{\I(y-x)}.
\]
Then setting 
\begin{gather}
\partial_k h = \sum_{n=0}^\infty g_n,\quad  g_0(k,x)
	:=\partial_k{h}_l(k,x),\label{eq:gser}
	\\
\begin{multlined}	
g_{n}(k,x)
	:=\int_x^\infty \partial_k \ti{G}_l(k,x,y)\,  {h}_{n-1}(k,y) q(y)dy
		\qquad\qquad\qquad\qquad
		\\
		+ \int_x^\infty \ti{G}_l(k,x,y) g_{n-1}(k,y) q(y)dy,
		\quad n\in \N,\label{eq:g_rec},
\end{multlined}		 
\end{gather}
we need to show that it satisfies the integral equation
\begin{multline}
\label{eq:partial-h}
\partial_k {h}(k,x)
= \partial_k{h}_l(k,x)\\
 	- \int_x^\infty [\partial_k \ti{G}_l(k,x,y)] {h}(k,y) + \ti{G}_l(k,x,y) \partial_k{h}(k,y)]q(y)dy. 
\end{multline}
 It easily follows from  \eqref{eq:partial-z-Gl} that
\[
\abs{\partial_k \ti{G}_l(k,x,y)} \le C |k|y \left(\frac{|k|y}{1+|k|y}\right)^{l+2}\left(\frac{1+|k|x}{|k|x}\right)^{l}, \quad 0<x \le y.
\]
Therefore, using \eqref{eq:f_n}, we can bound the first summand in 
\eqref{eq:g_rec} as follows
\begin{align*}
\abs{\text{1st term}}
&\le
\frac{C^{n+1}}{(n-1)!} 
	\left(\frac{1+|k|x}{|k|x}\right)^{l} 
	\int_x^\infty  \frac{|k|y^3\abs{q(y)}}{(1+|k|y)^2}\left(\int_y^\infty 
	\frac{t\abs{q(t)}}{1+|k|t}dt\right)^{n-1}dy
	\\
&\le
\frac{C^{n+1}}{n!}  \frac{1}{|k|}\left(\frac{1+|k|x}{|k|x}\right)^{l} 
	\left(\int_x^\infty y |q(y)| dy\right)^{n}.
\end{align*}
Next, using induction, one can show that the second summand admits a similar bound and hence we finally get
\[
\abs{g_n(k,x)}
\le \frac{C^{n+1}}{n!} \frac{1}{k} \left(\frac{1+|k|x}{|k|x}\right)^{l} 
	\left(\int_x^\infty  {y |q(y)|} dy\right)^{n}.
\] 
This immediately implies the convergence of 
\eqref{eq:gser} and, moreover, the estimate
\[
\abss{\partial_k{h}(k,x) -\partial_k{h}_l(k,x)}
\le \sum_{n=1}^\infty\abs{g_n(k,x)},
\]
from which \eqref{est:hk} follows under the assumption \eqref{eq:q_mar}.
\end{proof}

Furthermore, by \cite{fad, cc}, the Jost solution $f$ admits a representation by means of transformation operators preserving the behavior of solutions at infinity (see also \cite[Chap. V]{cs} for further details and historical remarks). 

\begin{lemma}[\cite{cc}]\label{lem:to}
Suppose  $\int_1^\infty (x+x^l)|q(x)| dx <\infty$. Then 
\be\label{eq:to_Mar}
f(k,x) = f_l(k,x) + \int_x^\infty K(x,y) f_l(k,y) dy =: (I+K)f_l(k,x),
\ee
where the so-called Marchenko kernel $K:\R^2\to \R$ satisfies the estimate
\be\label{eq:MAest}
|K(x,y)| \le \frac{1}{2} \left(\frac{2}{x}\right)^{l}\ti{\sigma}_l\left(\frac{x+y}{2}\right)\E^{\ti{\sigma}_1(x)},\quad \ti{\sigma}_j(x):=\int_x^\infty y^j |q(y)|dy,
\ee
for all $x<y<\infty$ and $j\in \{1,l\}$.
\end{lemma}

In particular, this lemma immediately implies the following useful result.

\begin{corollary}\label{cor:Kest}
Suppose $\int_1^\infty (x+x^{l+1})|q(x)| dx <\infty$. Then $K$ is a bounded operator on $L^\infty((1,\infty))$.
\end{corollary}

\begin{proof}
If $f\in L^\infty(\R_+)$, then using the estimate \eqref{eq:MAest} we get
\begin{align*}
|(K f)(x)| =& \Big|\int_x^\infty K(x,y) f(y) dy\Big| \le \|f\|_\infty \int_x^\infty |K(x,y)|dy \\
&\le \frac{1}{2}\|f\|_\infty \E^{\ti{\sigma}_1(1)} \left(\frac{2}{x}\right)^{l} \int_x^\infty \ti{\sigma}_l\Big(\frac{x+y}{2}\Big)dy\\
&\le 2^{l-1}  \|f\|_\infty \E^{\ti{\sigma}_1(1)} \int_1^\infty \int_{(1+y)/2}^\infty t^l |q(t)|dt\,dy \\
&= 2^{l-1}  \|f\|_\infty \E^{\ti{\sigma}_1(1)} \int_1^\infty \int_1^{2t - 1}  t^l |q(t)|dy\,dt \le 2^l\|f\|_\infty \ti{\sigma}_{l+1}(1)\E^{\ti{\sigma}_1(1)},
\end{align*}
which proves the claim.
\end{proof}

\subsection{The Jost function} 
By Lemma \ref{lem:b.4}, the Jost solution is analytic in the upper half plane and can 
be continuously extended to the real axis away from $k=0$. We can 
extend it to the lower half plane by setting 
$f(k,x) =f(-k,x)= f(k^*,x)^*$ for $\im(k)<0$.  For $k\in\R\setminus\{0\}$ we obtain two solutions $f(k,x)$ and 
$f(-k,x)=f(k,x)^*$ of the same equation whose Wronskian is given by (cf. \eqref{eq:JostSol})
\be\label{eq:wrfkpm}
W(f(-k,.),f(k,.))= 2\I k.
\ee
{\em The Jost function} is defined as
\be\label{eq:JostFunct}
f(k) = W(f(k,.),\phi(k^2,.))
\ee
and we also set
\[
g(k) = W(f(k,.),\theta(k^2,.))
\]
such that
\be
\label{eq:5.7}
f(k,x) = f(k) \theta(k^2,x) - g(k) \phi(k^2,x) = f(k) \psi(k^2,x).
\ee
In particular, the Weyl $m$-function \eqref{eq:s_m} is given by
\[
m(k^2) = -\frac{g(k)}{f(k)},\quad k\in\C_+.
\]
Note that both $f(k)$ and $g(k)$ are analytic in the upper half plane 
and $f(k)$ has simple zeros at $\I \kappa_n = \sqrt{\lam_n}\in\C_+$. 

Since $f(k,x)^*=f(-k,x)$ for $k\in\R\setminus\{0\}$, we obtain 
$f(k)^*=f(-k)$ and $g(k)^*=g(-k)$. Moreover, \eqref{eq:wrfkpm} shows
\be\label{eq:phif}
\phi(k^2,x) 
	= \frac{f(-k)}{2\I k} f(k,x) - \frac{f(k)}{2\I k} f(-k,x), 
	  \quad k\in\R\setminus\{0\},
\ee
and by \eqref{eq:5.7} we get
\[
2\I \im(f(k) g(k)^*)
	= f(k)g(k)^* -  f(k)^* g(k) 
	= W(f(-k,\cdot),f(k,\cdot))= 2\I k.
\]
Hence
\be\label{eq:imm=f}
\im\,m(k^2) = - \frac{\im\big(f(k)^*g(k)\big)}{|f(k)|^2} 
			= \frac{k}{\abs{f(k)}^2}, \quad k\in\R\setminus\{0\},
\ee
implying
\be\label{eq:rho}
d\rho(\lam) = \id_{(0,\infty)}(\lam) 
	\frac{\sqrt{\lam}}{\pi\abss{f(\sqrt{\lam})}^2} d\lam + 
	\sum_n \gam_n d\theta(\lam-\lam_n),
\ee
where $\gam_n^{-1} = \|\phi(\lam_n,\cdot)\|^2_{L^(\R_+)}$
are the usual norming constants. Since $-\gam_n$ equals the residue 
of $m(z)$ at $\lam_n$ we obtain
\[
\dot{f}(\I\kappa_n) 
	= -2\I\kappa_n\frac{g(\I\kappa_n)}{\gam_n}, \qquad f(\I\kappa_n,x)
	= g(\I\kappa_n) \phi(\lam_n,x) .
\]
Note that 
\[
f_l(k):= W(f_l(k,.),\phi_l(k^2,.)) = {k^{-l}}\E^{\I\frac{\pi l}{2}}, \quad 0\le \arg(k) < \pi.
\]
Thus, by Theorem \ref{th:main} and \eqref{eq:rho}, 
\be\label{eq:JFasymp}
\abs{f(k)}={\abs{f_l(k)}}(1+o(1))={\abs{k}^{-l}}(1+o(1)),\quad k\to \infty.
\ee

Finally, consider the following function
\be\label{eq:a.F}
F(k) := \E^{-\I\frac{\pi l}{2}}\, k^l f(k) = \frac{f(k)}{f_l(k)}  = \E^{-\I\frac{\pi l}{2}}\, k^l W(f(k,.),\phi(k^2,.)),\quad 
	  \im\, k\ge 0. 
\ee
Note that if we use
\be\label{eq:tipsi}
\ti\psi(k,x) = \frac{f(k,x)}{f_l(k)} = \E^{-\I\frac{\pi l}{2}} k^l f(k,x)
\ee
instead of $f(k,x)$, then $\ti\psi(\cdot,x)$ is analytic in the upper half plane and can 
be continuously extended to the whole real axis and \eqref{estpsi} now reads
\be\label{eq:tipsiest}
\abss{\ti\psi(k,x) - \psi_l(k,x)}
	\leq C \left(\frac{x}{1+ \abs{k} x}\right)^{-l} 
	\E^{-\abs{\im\, k}\, x} \int_x^\infty 
	\frac{y q(y)}{1 +\abs{k} y} dy.
\ee

\begin{lemma}[\cite{kt2}]\label{lem:b.5}
Assume \eqref{q:hyp} and \eqref{eq:q_mar}. Then the function $F$ admits 
the following integral representation
\begin{align}
F(k) &= 1+\int_0^\infty \psi_l(k^2,x)\phi(k^2,x) q(x)dx,\label{eq:intr_F}
	 \\	 
     &= 1+ \E^{-\I\frac{\pi l}{2}}\, k^l 
     	\int_0^\infty f(k,x)\phi_l(k^2,x) q(x)dx.\nn
\end{align}
If $l>-1/2$, then $F$ is analytic in $\C_+$, continuous and bounded on $\im\, k\ge 0$ and 
\be\label{eq:F=1}
F(k)=1+o(1)
\ee
as $\abs{k}\to \infty$ in $\im\, k\ge 0$. 
\end{lemma}

\begin{remark}
Note that $\phi(k^2,x)$ and $\phi_l(k^2,x)$ have the same leading 
asymptotics as $\abs{k}\to\infty$. Also,
\[
2k\,\psi_l(k^2,x)\,\phi_l(k^2,x) \to 1, \quad |k|\to\infty.
\]
Since $\abs{k\,\psi_l(k^2,x)\,\phi(k^2,x)} \le C$, 
dominated convergence implies that
\be\label{eq:FestRefined}
F(k) = 1 + \frac{\I}{2k} \int_0^\infty q(x) dx +o(k^{-1})
\ee
as $|k|\to \infty$ provided that $q\in L^1(0,\infty)$.
\end{remark}

We also will need the behavior of $F$ and $F'$ near zero. The next lemma is well known and we give its proof for the sake of completeness.

\begin{lemma}\label{lem:resonance}
Let $l>-1/2$ and assume \eqref{q:hyp} and \eqref{eq:q_mar}. The following conditions are equivalent:
\begin{enumerate}[label=\emph{(\roman*)}, ref=(\roman*), leftmargin=*, widest=ii]
\item $F(0)=0$,
\item  $\phi(0,.)$ and $\ti{\psi}(0,.)$ are linearly dependent,
\item  $\phi(0,x) \sim C x^{-l}$ as $x\to \infty$,
\item there is either a resonance (if $l\in (-1/2,1/2]$) or an eigenvalue (if $l>1/2$). 
\end{enumerate}
\end{lemma}

\begin{proof}
By \eqref{eq:a.F} and \eqref{eq:tipsi}, $F(0) = W(\ti{\psi}(0,.),\phi(0,.))$ which proves the equivalence $(i) \Leftrightarrow (ii)$. Moreover, the latter is further equivalent to $(iii)$ since $\ti{\psi}(0,x) \sim C x^{-l}$ as $x\to \infty$ in view of \eqref{eq:tipsiest}. 

Finally, the kernel of the resolvent $\cRH(k^2)=(H-k^2)^{-1}$ is given by
\be\label{eq:2.47}
[\cRH(k^2 + \I 0)](y,x)
        = \frac{\phi(k^2,x) f(k,y)}{f(k)}
	= \frac{\phi(k^2,x) \ti{\psi}(k,y)}{F(k)},\quad x\le y,
\ee
and hence we see that there is a resonance or an eigenvalue, i.e.\ a singularity of this kernel at $k=0$ if and
only if $F(0)=0$. Moreover,  this is also equivalent to existence of a solution
which, is bounded if $l\ge0$ and which is square integrable (i.e.\ an eigenfunction) if $l>\frac{1}{2}$.
\end{proof}

\begin{lemma}
\label{lem:about-F}
Assume \eqref{q:hyp} and \eqref{eq:q_mar}. Then $F(k)\ne 0$ for 
$k\in\R\setminus \{0\}$ and
\[
\abs{F(k)}^{-1} \le \OO(\abs{k}^{-\min(l+3/2,2)}),\qquad k\to 0.
\]
\end{lemma}
\begin{proof}
Since $f(k,x)$ can only be a multiple of $\phi(k^2,x)$ if $k=0$, their
Wronskian $f(k)$ can only vanish at $0$. Moreover, the singular Weyl
function must satisfy
\[
\abs{m(z)} \le \frac{C_\lam}{\abs{z-\lam}}
\]
near every $\lam\in\R$, which follows from its integral
representation \eqref{eq:int_rep}. Hence we obtain from  \eqref{eq:a.F} and \eqref{eq:imm=f}
\[
\frac{1}{\abs{F(k)}} =  \frac{1}{|k|^l \abs{f(k)}} = \frac{1}{|k|^{l+\frac12}}\sqrt{\im\, m(k^2)}  \le \frac{C}{\abs{k}^{l+\frac32}}
\]
as claimed.

To obtain the second bound we use that fact that the diagonal of the Green's function $\cRH(\cdot)(x,x)$ is a Herglotz--Nevanlinna function
and hence satisfies 
\[
\big|\cRH(z)(x,x)\big| \le \frac{C_{\lambda,x}}{|z - \lambda|}
\]
near every $\lambda\in\R$.
Choosing $x>0$ such that $\phi(0,x)\ti{\psi}(0,x)\neq 0$ and using \eqref{eq:2.47} we prove the claim.
\end{proof}

\begin{lemma}\label{lem:Fp}
Assume \eqref{q:hyp} and \eqref{eq:q_mar}. Then $F$ is differentiable for all $k\neq 0$ and 
\[
\abs{F'(k)}\le \frac{C}{1+\abs{k}}
\]
for all $|k|$ large enough. 
If in addition 
\be\label{eq:secmoment}
\int_1^\infty x^2 |q(x)|dx<\infty.
\ee
 then
\[
\abs{F'(k)} = \OO(|k|^{\min(0,2l)} ),\quad |k|\to 0.
\]
\end{lemma} 
\begin{proof}
Using \eqref{eq:intr_F}, we get
\[
F'(k) 
= \int_0^\infty \big(\partial_k\psi_l(k^2,x)\phi(k^2,x) 
  + \psi_l(k^2,x)\partial_k\phi(k^2,x) \big)q(x)dx .  
\]
The integral converges absolutely for all $k\neq 0$. Indeed, by \eqref{est:psi_lA} and \eqref{eq:partial-z-diff-phi}, we obtain
\begin{align*}
\left|\int_0^\infty \psi_l(k^2,x)\partial_k\phi(k^2,x) q(x)dx \right| \le C  \int_0^\infty \frac{x^2}{1+|k|x}|q(x)|dx
\end{align*}
Using \eqref{estphi} and \eqref{eq:partial-z-weyl-sol-free}, 
we get the following estimates for the first summand:
\begin{align*}
\abs{\int_0^\infty\!\!\!\partial_k\psi_l(k^2,x)\,\phi(k^2,x) q(x)dx} 
	\le C \begin{dcases}
          \int_0^\infty\!\!\! 
          \frac{x^2}{1+|k|x}|q(x)|dx, & l\ge 0,
          \\
          |k|^{2l}\int_0^\infty\!\!\! 
          \frac{x^{2+2l}}{(1+|k|x)^{1+2l}}|q(x)|dx, 
          	& l\in (-\tfrac{1}{2},0).
          \end{dcases}
\end{align*}
Now the claim follows.
\end{proof}

%%%%%%%%%%%%%%%%%%%%%%%%%%%%%%%%%%%%%%%%%%%%%%%%%%%%%%%%%%%%%%%%%%%%%%%%%
\section{Dispersive decay}
\label{ll-sec}
%%%%%%%%%%%%%%%%%%%%%%%%%%%%%%%%%%%%%%%%%%%%%%%%%%%%%%%%%%%%%%%%%%%%%%%%%

In this section we prove the dispersive decay estimate \eqref{full} for 
the Schr\"odinger equation \eqref{Schr}. In order to do this, we divide
the analysis into a low and high energy regimes. In the analysis of both regimes we make use of variants of the 
van der Corput lemma (see Appendix \ref{sec:vdCorput}), combined with a Born series approach
for the high energy regime. 

\subsection{The low energy part}
\label{ll-sec-low}

For the low energy
regime, it is convenient to use the following well-known
representation of the integral kernel of $\E^{-\I tH}P_{c}(H)$,
\begin{align}
[\E^{-\I tH}P_{c}(H)](x,y)
	&= \frac{2}{\pi} \int_{-\infty}^{\infty}
       \E^{-\I t k^2}\,\phi(k^2,x) \phi(k^2,y) \im m(k^2) k \,dk\nn
       \\
	&=\frac{2}{\pi} \int_{-\infty}^{\infty}
	  \E^{-\I t k^2}\,\frac{\phi(k^2,x) \phi(k^2,y)k^2}{|f(k)|^2}\,dk\nn
	  \\
	&=\frac{2}{\pi} \int_{-\infty}^{\infty}
	  \E^{-\I t k^2}\,\frac{\phi(k^2,x)\phi(k^2,y)|k|^{2(l+1)}}{|F(k)|^2} 
	  \,dk\label{integr}
	  \\
	&=\frac{2}{\pi} \int_{-\infty}^{\infty}
	  \E^{-\I t k^2}\,\frac{\tilde{\phi}(k,x)\tilde{\phi}(k,y)}
	  {|F(k)|^2} \,dk,\nn
\end{align}
where the integral is to be understood as an improper integral. In fact, adding an additional
energy cut-off (which is all we will need below) the formula is immediate from the spectral
transformation \cite[\S3]{kst2} and the general case can then be established taking limits
using the bounds on this kernel to be established below.
In the last equality we have used  
\begin{equation}\label{eq:tiphif}
\tilde{\phi}(k,x)
	:= |k|^{l+1} \phi(k^2,x),\quad k\in \R. 
\end{equation}
Note that 
\begin{align}\label{eq:esttiphi}
|\ti{\phi}(k,x)| &\le C \left(\frac{|k|x}{1+|k|x}\right)^{l+1}\E^{|\im k|x},\\
\label{eq:esttiphik}
|\partial_k \ti{\phi}(k,x)| &\le Cx \left(\frac{|k|x}{1+|k|x}\right)^{l}\E^{|\im k|x},
\end{align}
which follows  from \eqref{estphil}, \eqref{estphi} and
\[
\partial_k \ti{\phi}(k,x) = (l+1){\rm sgn}(k)|k|^l \phi(k^2,x) + |k|^{l+1}\partial_k \phi(k^2,x)
\]
together with \eqref{eq:partial-z-phil}, \eqref{eq:partial-z-diff-phi}.

We begin with the following estimate.

\begin{theorem}\label{thm:le1}
Assume \eqref{q:hyp} and \eqref{eq:secmoment}. Let also $\chi\in C^\infty_c(\R)$ with ${\rm supp}(\chi)\subset (-k_0,k_0)$
and suppose there is neither a resonance nor an eigenvalue at $0$, that is $F(0)\ne 0$. Then
\be\label{eq:low-energy-1}
\big|[\E^{-\I tH}\chi(H)P_{c}(H) ](x,y)\big| \le \frac{C}{|t|^{1/2}} \max(x,y).
\ee
\end{theorem}

\begin{proof}
We want to apply the van der Corput Lemma \ref{lem:vC} with $c=0$ and
\[
A(k) = \chi(k^2) A_0(k), \qquad A_0(k) = \frac{\phi(k^2,x) \phi(k^2,y)\abs{k}^{2(l+1)}}{\abs{F(k)}^2}=\frac{\tilde{\phi}(k,x)\tilde{\phi}(k,y)}{|F(k)|^2}.
\]
Note that 
\[
\|A\|_\infty \le \|\chi\|_\infty\|A_0\|_\infty,\qquad \|A'\|_1 \le \|\chi'\|_1\|A_0\|_\infty + \|\chi\|_\infty\|A_0'\|_1.
\]

Our assumption $F(0)\ne 0$ together with Lemma \ref{lem:about-F} imply  $F(k)\ne 0$ for all $k\in\R$ and hence $\|1/F\|_\infty<\infty$ in view of Lemma~\ref{lem:b.5}.
Using \eqref{eq:esttiphi} we infer
\be\label{eq:estAinfty}
\sup_{k^2\le k_0} |A_0(k)| \le C\,\|1/F\|^2_\infty \big(\min(1, k_0xy)\big)^{l+1}<\infty,
\ee
which holds for all $x$ and $y$ with some uniform constant $C>0$. 
Moreover,
\[
A_0'(k) = \frac{\partial_k\tilde{\phi}(k,x)\tilde{\phi}(k,y) + \tilde{\phi}(k,x)\partial_k\tilde{\phi}(k,y)}
		{\abs{F(k)}^2}
		-
		 A_0(k) \re\frac{F'(k)}{F(k)}
\]
and it suffices to bound the two terms from above on compact sets.
In fact, it suffices to consider the first term since the second one follows
from~ \eqref{eq:estAinfty} and Lemma~\ref{lem:Fp}.

The estimate for the first term follows from \eqref{eq:esttiphi} and \eqref{eq:esttiphik} since 
\begin{align*}
&\left|\partial_k\tilde{\phi}(k,x)\tilde{\phi}(k,y) + \tilde{\phi}(k,x)\partial_k\tilde{\phi}(k,y)\right| \\
&\quad\quad \le C \left(\frac{|k|x}{1+|k|x}\right)^{l+1} \left(\frac{|k|y}{1+|k|y}\right)^{l+1} \left( \frac{1+|k|x}{|k|} + \frac{1+|k|y}{|k|}\right)\\
&\quad\quad \le C \sqrt{xy} \left(\frac{|k|x}{1+|k|x}\right)^{l+1/2} \left(\frac{|k|y}{1+|k|y}\right)^{l+1/2} \frac{1+|k|x + 1+ |k|y}{\sqrt{(1+|k|x)(1+|k|y)}}\\
& \quad\quad \le C \sqrt{xy} \left(\sqrt{\frac{1+|k|x}{1+|k|y}} + \sqrt{\frac{1+|k|y}{1+|k|x}}\right) \le C\max(x,y).
\end{align*}
It remains to apply the van der Corput lemma. 
\end{proof}

To get rid of the dependence of $x$ and $y$ in \eqref{eq:low-energy-1} we make
use of the transformation operators \eqref{eq:to_GL} and \eqref{eq:to_Mar}.

\begin{theorem}\label{thm:le2}
Assume
\be
\int_0^1 |q(x)| dx <\infty \quad \text{and} \quad  \int_1^\infty x^{\max(2,l+1)} |q(x)| dx <\infty.
\ee
Let also $\chi\in C^\infty_c(\R)$ with ${\rm supp}(\chi)\subset (-k_0,k_0)$ and
suppose there is neither a resonance nor an eigenvalue at $0$, that is $F(0)\ne 0$.
Then
\be\label{eq:3.11}
\big|[\E^{-\I tH}\chi(H)P_{c}(H) ](x,y)\big| \le \frac{C}{|t|^{1/2}}, \qquad \max(x,y)\ge 1.
\ee
\end{theorem}

\begin{proof}
Assume that $ 0<x\le 1\le y$. 
We proceed as in the previous proof but use Lemma~\ref{lem:toGL} and Lemma~\ref{lem:to} to write
\[
A(k) = \chi(k^2) \frac{(I+B_x) \ti{\phi}_l(k,x) \cdot (I+K_y) \ti{\phi}_l(k,y)}{|F(k)|^2},\quad k\neq 0.
\]
Indeed, for all $k\in \R\setminus\{0\}$, $\phi(k,\cdot)$ is bounded at infinity and admits the representation \eqref{eq:phif} by means of Jost solution $f(k,\cdot)$ and $f(-k,\cdot)$. Therefore, by Lemma~\ref{lem:to}, $\ti{\phi}(k,y) = (I+K_y)\ti{\phi}_l(k,y)$ for all $k\in \R\setminus\{0\}$. 
 
By symmetry $A(k)=A(-k)$ and hence our integral reads
\[
I(t,x,y)=\frac{4}{\pi} \int_0^\infty \E^{-\I t k^2} A(k) dk.
\]
Our aim is to use Lemma~\ref{lem:vC2} (plus the remarks after this lemma) and
hence we need to show that the individual parts of $A(k)$ coincide with a function
which is the Fourier transform of a finite measure. In particular, we can
redefine $A(k)$ for $k<0$. To this end note that $\ti{\phi}_l(k^2,x)= J(|k| x)$, where
\[
J(r) = \sqrt{r}\, J_{l+\frac{1}{2}}(r) = \frac{r^{l+1}}{2^{l+1/2}} \sum_{n=0}^\infty 
	\frac{(-r^2/4)^n}{n!\Gamma(\nu+n+1)}, \quad r\ge 0.
\]
Note that $J(r)\sim r^{l+1}$ as $r\to 0$ and $J(r)= \sqrt{\frac{2}{\pi}} \sin(r-\frac{\pi l}{2}) +O(r^{-1})$
as $r\to+\infty$ (see \eqref{eq:asymp-J-infty}). Moreover, $J'(r)\sim r^l$ as $r\to 0$ and $J'(r)= \sqrt{\frac{2}{\pi}} \cos(r-\frac{\pi l}{2}) +O(r^{-1})$ as $r\to+\infty$ (see \eqref{eq:a11}). In particular, $\ti{J}(r):=J(r)-\sqrt{\frac{2}{\pi}} \sin(r-\frac{\pi l}{2})$ is in $H^1(\R_+)$.
Moreover, we can define $J(r)$ for $r<0$ such that it is locally in $H^1$ and $J(r)=\sqrt{\frac{2}{\pi}} \sin(r-\frac{\pi l}{2})$
for $r<-1$. By construction we then have $\ti{J}\in H^1(\R)$ and thus $\ti{J}$ is the Fourier transform of an integrable
function. Moreover, $\sin(r-\frac{\pi l}{2})$ is the Fourier transform of the sum of two Dirac delta measures and so
$J$ is the Fourier transform of a finite measure.
By scaling, the total variation of the measures corresponding to $J(k x)$ is independent of $x$.
Since the same is true for $\chi(k^2) |F(k)|^{-2}$ by Lemma~\ref{lem:Fp}, an application of
Lemma~\ref{lem:vC2} shows
\[
|\ti{I}(t,x,y)| \le \frac{C}{\sqrt{t}}, \qquad \ti{I}(t,x,y)=\frac{4}{\pi} \int_0^\infty \E^{-\I t k^2} \chi(k^2) \frac{\ti{\phi}_l(k,x) \ti{\phi}_l(k,y)}{|F(k)|^2} dk.
\]
But by Fubini we have $I(t,x,y)= (1+B_x)(1+K_y)\ti{I}(t,x,y)$ and the claim follows since both $B:L^\infty((0,1))\to L^\infty((0,1))$ and $K:L^\infty((1,\infty)) \to L^\infty((1,\infty))$ are
bounded in view of Corollary \ref{cor:Best} and Corollary \ref{cor:Kest}, respectively.

By symmetry, we immediately obtain the same estimate if $0<y\le 1\le x $. The case $\min(x,y)\ge 1$ can be proved analogously, we only need to to write
\[
A(k) = \chi(k^2) \frac{(I+K_x) \ti{\phi}_l(k,x) \cdot (I+K_y) \ti{\phi}_l(k,y)}{|F(k)|^2},\quad k\neq 0.\qedhere
\]
\end{proof}

%%%%%%%%%%%%%%%%%%%%%%%%%%%%%%%%%%%%%%%%%%%%%%%%%%%%%%%%%%%%%%%%%%%%%%%%%%%%%

\subsection{The high energy part}
For the analysis of the high energy regime we use the following 
---also well known--- alternative representation:

\begin{align}
\E^{-\I tH}P_{c}(H)
	&=\frac 1{2\pi \I}\int_{0}^{\infty}\E^{-\I t\omega}
	  \left[\cRH(\omega+\I 0)-\cRH(\omega-\I 0)\right] d\omega\nn
	  \\
	&=\frac 1{\pi \I}\int_{-\infty}^{\infty}\E^{-\I t k^2}
	  \cRH(k^2+\I 0)\,k\,dk,\label{PP}	 
\end{align}
where $\cRH(\omega)=(H-\omega)^{-1}$ is the resolvent of the 
Schr\"odinger operator $H$ and the limit is understood in the strong 
sense \cite{tschroe}. We recall that the Green's function is given by
\[
[\cRH(k^2\pm \I 0)](x,y) = [\cRH(k^2\pm \I 0)](y,x)
	= \phi(k^2,x) \psi(k^2\pm \I0,y),\quad x\le y.
\]
Note also that 
\[
\psi(k^2\pm \I0,x) 
	= \frac{f(\pm k,x)}{f(\pm k)},\quad k\in\R\setminus\{0\}.  
\]

Fix $k_0 > 0$ and let $\chi:\R\to [0,\infty)$ be a $C^\infty$ function such that
\be\label{defchi0}
\chi(k^2) = \begin{cases}
0, & \abs{k}< 2k_0,\\
1, & \abs{k}> 3k_0.
\end{cases}
\ee 
The purpose of this section is to prove the following estimate.

\begin{theorem}\label{thm:he}
Suppose $q\in L^1(\R_+)$. Then
\[
\big|[\E^{-\I tH}\chi(H)P_{c}(H) ](x,y)\big| \le \frac{C}{|t|^{1/2}}.
\]
\end{theorem}

Our starting point is the fact that the resolvent $\cRH$ of $H$ can be expanded into the Born series
\be\label{eq:born}
\cRH(k^2\pm \I0) = \sum_{n=0}^\infty \cRl(k^2\pm \I0)(-q\,\cRl(k^2\pm \I0))^n,
\ee 
where $\cRl$ stands for the resolvent of the unperturbed radial Schr\"odinger
operator.

To this end we begin by collecting some facts about $\cRl$. Its kernel is given
\[
\cRl(k^2\pm \I0,x,y) = \frac{1}{k} r_l(\pm k,x,y),
\]
where
\[
r_l(k;x,y)=r_l(k;y,x):= k\sqrt{xy}\,J_{l+\frac{1}{2}}(kx)
			 H_{l+\frac{1}{2}}^{(1)}(ky),\quad x\le y.
\]

\begin{lemma}\label{lem:rl}
The function $r_l(k,x,y)$, $l>-\frac{1}{2}$, can be written as
\[
r_l(k,x,y) 
	= \chi_{(-\infty,0]}(k) 
	  \int_\R \E^{\I k p} d\rho_{l,x,y}(p) 
	  + \chi_{[0,\infty)}(k) \int_\R \E^{-\I k p} d\rho_{l,x,y}^\ast(p)
\]
with a measure whose total variation satisfies
\[
\|\rho_{l,x,y}\| \le C(l).
\]
Here $\rho^\ast$ is the complex conjugated measure.
\end{lemma}

\begin{proof}
Let $x\le y$ and $k\ge 0$. Write
\[
r_l(k,x,y)= \frac{\chi_l(k x)}{\chi_l(k y)} J(k x) H(k y),
\]
where
\[
J(r) =  \chi_l(r)^{-1} \sqrt{r}\, J_{l+\frac{1}{2}}(r), 
\quad
H(r) =  \chi_l(r) \sqrt{r}\, H_{l+\frac{1}{2}}^{(1)}(r), 
\quad 
\chi_l(r)=\left(\frac{r^2}{1+r^2}\right)^{(l+1)/2}.
\]
We continue $J(r)$, $H(r)$ to the region $r<0$ such that they are 
continuously differentiable and satisfy 
\[
J(r)=\sqrt{\frac{2}{\pi}} \sin\left(r-\frac{\pi l}{2}\right),\quad
H(r)=\sqrt{\frac{2}{\pi}} \E^{\I\left(r-\frac{\pi(l+1)}{2}\right)},
\]
for $r<-1$. Then $\ti{J}(r):=J(r)-\sqrt{\frac{2}{\pi}} \sin(r-\frac{\pi l}{2})$ 
and $\ti{H}(r):=H(r)-\sqrt{\frac{2}{\pi}} \E^{\I(r-\frac{\pi(l+1)}{2})}$ are in 
$H^1(\R)$. In fact, they are continuously differentiable and hence 
it suffices to look at their asymptotic behavior. For $r<-1$ they are 
zero and for $r>1$ they are $O(r^{-1})$ and their derivative is 
$O(r^{-1})$ as can be seen from the asymptotic behavior of the Bessel 
and Hankel functions (see Appendix \ref{sec:Bessel}). Hence both $J$ and $H$ are Fourier 
transforms of finite measures. By scaling the total variation of the 
measures corresponding to $J(k x)$, $H(k y)$ are independent of $x$, 
$y$, respectively.

Hence it remains to consider the Fourier transform
\[
F_{x,y}(p)
 := \sqrt{\frac{2}{\pi}}\int_0^\infty 
 	\left(1- \frac{\chi_l(k x)}{\chi_l(k y)} \right)\cos(kp) dk.
\]
First observe that 
\[
F_{x,y}(p) = \frac{1}{x} F_{1,y/x}(p/x). 
\]
for all $x\le y$. Therefore, $\|F_{x,y}\|_{L^1} = \|F_{1,y/x}\|_{L^1}$.  Hence it suffices to consider the Fourier transform of 
\[
h_{\eta,l}(k) := 1 - \left(\frac{\eta+k^2}{1+k^2}\right)^{(l+1)/2}, 
\qquad 
\eta   := \frac{x^2}{y^2} \in(0, 1].
\]
First, note that 
\[
h_{\eta,l}'(k) = {(l+1)(\eta-1)}\left(\frac{\eta+k^2}{1+k^2}\right)^{(l-1)/2} \frac{k}{(1+k^2)^2},\quad k\in\R.
\]
Therefore,
\[
h_{\eta,l}(k) = \frac{(l+1)(1-\eta)}{2k^2}(1+o(k)),\quad h_{\eta,l}'(k) = \frac{(l+1)(\eta-1)}{k^3}(1+o(k))
\]
as $k\to \infty$. This implies that $h_{\eta,l} \in H^1(\R)$ and hence $\hat{h}_{\eta,l}\in L^1(\R)$. According to \eqref{eq:embedding}, it remains to show that the family $h_{\eta,l}$ is uniformly bounded in $H^1(\R)$ with respect to $\eta\in (0,1]$. Clearly,
\[
|h_{\eta,l}(k)| \le |h_{0,l}(k)|
\] 
for all $k\in\R$ and hence $\|h_{\eta,l}\|_{L^2}\le \|h_{0,l}\|_{L^2}$. Noting that 
\[
\left( \frac{\eta+ k^2}{1+k^2} \right)^{\frac{l-1}{2}} \le \begin{cases} 1, & l\ge 1 \\ \left(\frac{1+k^2}{k^2}\right)^{\frac{1-l}{2}}, & l\in (-\frac{1}{2},1),\end{cases}, \quad k>0,
\]
for all $l>-\frac{1}{2}$, we get
\[
|h_{\eta,l}'(k)| \le (l+1)\frac{|k|}{(1+k^2)^2} \begin{cases} 1, & l\ge 1, \\ \left(\frac{1+k^2}{k^2}\right)^{\frac{1-l}{2}}, & l\in (-\frac{1}{2},1).\end{cases}
\]
The latter implies that $\|h_{\eta,l}'\|_{L^2}$ are uniformly bounded.
\end{proof}

\begin{remark} 
%A few remarks a in order.
\begin{enumerate}[label=\emph{(\roman*)}, ref=(\roman*), leftmargin=*, widest=ii]
\item For $l\in\N_0$ the situation is somewhat simpler and we can write
\[
r_l(k,x,y) = \int_\R \E^{\I k p} d\rho_{l,x,y}(p), \qquad l\in\N_0,
\]
with
\begin{align*}
\rho_{l,x,y}(p) =& \sqrt{\frac{2}{\pi}} \left(\delta(p-x+y) -(-1)^l \delta(p+x+y) \right)\\
& {} +
\frac{\sign(p-x+y)-\sign(p+x+y)}{\sqrt{2\pi}} P_{l,x,y}(p)
\end{align*}
where $P_{l,x,y}(p)$ is a polynomial of degree $2l-1$ which is symmetric in $x$ and $y$.
Explicitly,
\begin{align*}
P_{0,x,y}(p) &= 0,\quad 
P_{1,x,y}(p) = - \frac{p}{xy},\quad 
P_{2,x,y}(p) = \frac{3 p \left(p^2-x^2-y^2\right)}{2 x^2 y^2}\\
P_{3,x,y}(p) &= -\frac{3 p \left(5 \left(p^2-x^2\right)^2+2 y^2 \left(3 x^2-5 p^2\right)+5
   y^4\right)}{8 x^3 y^3}
\end{align*}
and one can verify the claim explicitly.
\item We have the following recursion
\[
r_{l+1}(k,x,y)
	=   r_{l-1}(k,x,y) - \frac{2l+1}{k x y} \left(\frac{d}{dk} 
	  - \frac{1}{k}\right) r_l(k,x,y).
\]
\end{enumerate}
\end{remark}

Now we are in position to finish the proof of the main result.

\begin{proof}[Proof of Theorem~\ref{thm:he}]
As a consequence of Lemma~\ref{lem:rl} we note
\[
|\cRl(k^2\pm \I0,x,y)| \le \frac{C(l)}{|k|}
\]
and hence the operator $q\, \cRl(k^2\pm \I0)$ is bounded on $L^1$ with
\[
\|q\,\cRl(k^2\pm \I0)\|_{L^1}\le \frac{C(l)}{ |k|} \|q\|_{L^1}.
\]
Thus we get 
\begin{align*}
\abs{\inner{\cRl(k^2\pm \I0)(-q\, \cRl(k^2\pm \I0))^n f}{g}}
& =  \abs{\inner{-q\, \cRl(k^2\pm \I0))^n f}{\cRl(k^2\mp \I0)g}}
\\
&\le \norm{(-q\, \cRl(k^2\pm \I0))^n f}_{L^1}
	 \norm{\cRl(k^2\mp \I0)g}_{L^\infty}
\\
&\le \frac{C(l)^{n+1}\|q\|_{L^1}^n}{|k|^{n+1}}
	 \norm{f}_{L^1}\norm{g}_{L^1}.
\end{align*}
This estimate holds for all $L^1$ functions $f$ and $g$ and hence the 
series \eqref{eq:born} weakly converges whenever 
$\abs{k}>k_0:=C(l)\|q\|_{L^1}$. Namely, for all $L^1$ functions $f$ and $g$ 
we have
\be
\label{eq:born_weak}
\inner{\cRH(k^2\pm \I0)f}{g} 
	= \sum_{n=0}^\infty \inner{\cRl(k^2\pm \I0)
	  (-q\, \cRl(k^2\pm \I0))^nf}{g}.
\ee 
Using the estimates \eqref{estphi}, \eqref{estpsi}, \eqref{eq:a.F}, 
and \eqref{eq:F=1} for the Green's function of the perturbed operator, 
one can see that 
\[
\cRH(k^2\pm \I0)\, g\in L^\infty
\]
whenever $g\in L^1$ and $|k|>0$. Therefore, we get
\begin{multline*}
\abs{\inner{\cRH(k^2\pm \I0)(-q\, \cRl(k^2\pm \I0))^nf}{g}}
\\
\begin{aligned}
& =  \abs{\inner{(-q\, \cRl(k^2\pm \I0))^nf}{\cRH(k^2\mp \I0)g}}
\\
&\le \norm{(-q\, \cRl(k^2\pm \I0))^nf}_{L^1}
	 \norm{\cRH(k^2\mp \I0)g}_{L^\infty}
\\
&\le \left(\frac{C(l)\norm{q}_{L^1}}{k}\right)^n 
	 \norm{\cRH(k^2\mp \I0)g}_{L^\infty},
\end{aligned}	 
\end{multline*}
which means that $\cRH(k^2\pm \I0)(-q\, \cRl(k^2\pm \I0))^n$ 
weakly tends to 0 whenever $\abs{k}>k_0$. 

Let us consider again a function $\chi$ as in \eqref{defchi0} with $k_0:=C(l)\|q\|_1$.
Since $\E^{\I tH}\chi(H)P_c = \E^{\I tH}\chi(H)$, we get from 
\eqref{PP} 
\begin{align*}
\inner{\E^{-\I tH}\chi(H) f}{g} 
= \frac 1{\pi \I}\int_{-\infty}^\infty   
   \E^{-\I t k^2}\chi(k^2)k\inner{\cRH(k^2+\I 0)f}{g} dk.
\end{align*}
Using  \eqref{eq:born_weak} and noting that we can exchange summation and 
integration, we get
\begin{multline*}
\inner{\E^{-\I tH}\chi(H) f}{g} \\
	= \frac 1{\pi \I}\sum_{n=0}^{\infty} 
	  \int_{-\infty}^\infty\E^{-\I t k^2}\chi(k^2)k
	  \inner{\cRl(k^2+ \I0)(-q\, \cRl(k^2+ \I0))^nf}{g} dk.
\end{multline*}
The kernel of the operator $\cRl(k^2+ \I0)(-q\, \cRl(k^2+ \I0))^n$ 
is given by 
\[
\frac{1}{k^{n+1}} \int_{\R_+^n} r_l(k;x,y_1)\prod_{i=1}^{n} q(y_i)
	\prod_{i=1}^{n-1} r_l(k;y_i,y_{i+1}) r_l(k;y_n,y)dy_1\cdots dy_n.
\]  
Applying Fubini's theorem, we can integrate in $k$ first and hence we 
need to obtain a uniform estimate of the oscillatory integral
\[
I_n(t;u_0,\ldots,u_{n+1}) 
	:= \int_{\R} \E^{-\I t k^2}\chi(k^2)
	   \left(\frac{k}{2k_0}\right)^{-n}
	   \prod_{i=0}^{n}r_l(k;u_i,u_{i+1})\,dk
\]
since, recalling that $k_0=C(l)\|q\|_{L^1}$, one obtains
\begin{align*}
\abs{\inner{\E^{-\I tH}\chi(H) f}{g}}
	\le \frac{1}{\pi}\sum_{n=0}^{\infty} \frac{1}{(2C(l))^n}
		\!\!\sup_{\{u_i\}_{i=0}^{n+1}}
		\abss{I_n(t;u_0,\ldots,u_{n+1})}
		\norm{f}_{L^1}\norm{g}_{L^1}.
\end{align*}
Consider the function $f_n(k)=\chi(k^2) \left(\frac{k}{2k_0}\right)^{-n}$. Clearly, 
$f_0$ is the Fourier transform of a measure $\nu_0$ satisfying $\|\nu_0\| \le C$.
For $n\ge 1$, $f_n$ is $H^1$ with $\|f_n\|_{H^1} \le \pi^{-1/2} C(1+n)$. Hence
by Lemma~\ref{lem:vC2} and Lemma~\ref{lem:rl} we obtain
\[
\abss{I_n(t;u_0,\ldots,u_{n+1})}\le \frac{2 C_v C}{\sqrt{t}} (1+n) C(l)^{n+1}
\]
implying
\[
\abs{\inner{\E^{-\I tH}\chi(H) f}{g}}\\
	\le \frac{2 C_v C \, C(l)}{\sqrt{t}} \norm{f}_{L^1}\norm{g}_{L^1} \sum_{n=0}^{\infty} \frac{1+n}{2^n}.
\]
This proves Theorem~\ref{thm:he}.
\end{proof}

\appendix

\section{The van der Corput Lemma}\label{sec:vdCorput}

We will need the classical van der Corput lemma (see e.g.\ \cite[page 334]{St}):

\begin{lemma}\label{lem:vC}
Consider the oscillatory integral
\begin{equation*}
I(t) = \int_a^b \E^{\I t k^2 + \I c k} A(k) dk.
\end{equation*}
If $A\in\mathrm{AC}(a,b)$, then
\begin{equation*}
\abs{I(t)} 
	\le C_2 \abs{t}^{-1/2}(\norm{A}_\infty + \norm{A'}_1), 
	\quad \abs{t}\ge 1,
\end{equation*}
where $C_2\le 2^{8/3}$ is a universal constant.
\end{lemma}

Note that we can apply the above result with $(a,b)=(-\infty,\infty)$
by considering the limit $(-a,a)\to(-\infty,\infty)$.

Our proof will be based on the following variant of the van der Corput 
lemma which can be shown as in \cite[Lemma~5.1]{EKT}.

\begin{lemma}\label{lem:vC2}
Let $(a,b)\subseteq\R$ and consider the oscillatory integral
\[
I(t) = \int_a^b \E^{\I t k^2} A(k) dk.
\]
If $A$ is the Fourier transform of a signed measure
\[
A(k) = \int_\R \E^{\I k p} d\alpha(p),
\]
then the above integral exists as an improper integral and satisfies
\begin{equation*}
\abs{I(t)} 
	\le C_2 \abs{t}^{-1/2} \norm{\alpha}, 
	\quad \abs{t}>0.
\end{equation*}
where $\norm{\alpha}=\abs{\alpha}(\R)$ denotes the total variation of $\alpha$ and
$C_2$ is the constant from the van der Corput lemma.
\end{lemma}

In this respect we note that if $A_j$ are two such functions then 
(cf.\ p. 208 in \cite{bo})
\[
(A_1 A_2)(k)= \frac{1}{(2\pi)^2} \int_\R \E^{\I k p} d(\alpha_1*\alpha_2)(p)
\]
is associated with the convolution
\[
\alpha_1 * \alpha_2(\Omega) = \iint \id_\Omega(x+y) d\alpha_1(x) d\alpha_2(y),
\]
where $\id_\Omega$ is the indicator function of a set $\Omega$. 
Note that
\[
\|\alpha_1 * \alpha_2\| \le \|\alpha_1\| \|\alpha_2\|.
\]
We also need the following simple fact due to Beurling: If $f\in H^1(\R)$, then $f$ is in the Wiener algebra $\mathcal{A}$ and
\be\label{eq:embedding}
\|f\|_{\mathcal{A}}:= \|\hat{f}\|_{L^1(\R)} \le \sqrt{\pi} \|f\|_{H^1(\R)}.
\ee

\section{Bessel functions}\label{sec:Bessel}

Here we collect basic formulas and information on Bessel and Hankel functions (see, e.g., \cite{dlmf, Wat}).
We start with the definitions:
\begin{align}
J_{\nu}(z) 
&= \left(\frac{z}{2}\right)^\nu \sum_{n=0}^\infty 
	\frac{(-z^2/4)^n}{n!\Gamma(\nu+n+1)},\label{eq:Jnu01}
	\\
Y_{\nu}(z) 
&= \frac{J_{\nu}(z)\cos(\nu\pi) - J_{-\nu}(z)}{\sin(\nu \pi)},\label{eq:Ynu01}
	\\
H_{\nu}^{(1)}(z) 
&= J_{\nu}(z) + \I Y_{\nu}(z), \qquad
H_{\nu}^{(2)}(z) 
= J_{\nu}(z) - \I Y_{\nu}(z).
\end{align}
Note that the right-hand side in equation \eqref{eq:Ynu01} should be replaced by its limiting value if $\nu\in \Z$. Their asymptotic behavior as $|z|\to\infty$ is given by 
\begin{align}
J_{\nu}(z) 
	&= \sqrt{\frac{2}{\pi z}}\left(\cos(z- \nu \pi/2 - \pi/4) 
	   + \E^{|\im z|}\OO(|z|^{-1})\right),
	   \quad |\arg z|<\pi\label{eq:asymp-J-infty},
	\\
Y_{\nu}(z) 
	&= \sqrt{\frac{2}{\pi z}}\left(\sin(z- \nu \pi/2 - \pi/4) 
	   + \E^{|\im z|}\OO(|z|^{-1})\right),
	   \quad \abss{\arg z}<\pi\label{eq:asymp-Y-infty},
	 \\
H_{\nu}^{(1)}(z)
	&= \sqrt{\frac{2}{\pi z}}\E^{i(z-\frac{1}{2}\nu \pi 
	   - \frac{1}{4}\pi)}\left(1 + \OO(|z|^{-1})\right),
	   \quad -\pi<\arg z<2\pi,\label{eq:asymp-H1-infty},
	   \\
H_{\nu}^{(2)}(z)
	&= \sqrt{\frac{2}{\pi z}}\E^{-i(z-\frac{1}{2}\nu \pi 
	   - \frac{1}{4}\pi)}\left(1 + \OO(|z|^{-1})\right),
	   \quad -2\pi<\arg z<\pi.\label{eq:asymp-H2-infty}	    
\end{align}
Denote the reminder in \eqref{eq:asymp-J-infty}, \eqref{eq:asymp-Y-infty} and
\eqref{eq:asymp-H1-infty} by $j_l(z)$, $y_l(z)$ and $h_l(z)$, respectively. 
Noting that 
\be\label{eq:recurrence}
\mathcal{Y}_{\nu}'(z) = -\mathcal{Y}_{\nu+1}(z) + \frac{\nu}{z} \mathcal{Y}_{\nu}(z) =\mathcal{Y}_{\nu-1}(z) - \frac{\nu}{z} \mathcal{Y}_{\nu}(z) ,
\ee
one can show that the derivative of the reminder satisfies
\be\label{eq:a11}
\left(\sqrt{\frac{\pi z}{2}}J_{\nu}(z) - \cos(z-\frac{1}{2}\nu \pi - \frac{1}{4}\pi)\right)'= \E^{|\im z|}\OO(|z|^{-1}),
\ee
as $|z|\to \infty$. The same is true for $Y_\nu$, $H_\nu^{(1)}$ and $H_\nu^{(2)}$.

According to \cite[formula (VI.12.3)]{Wat}, for real non-zero $k$
\begin{align*}
H_{l+\frac{1}{2}}^{(1)}(k) 
	&= \sqrt{\frac{2}{\pi k}} \frac{\E^{\I (k -\pi(l+1)/2)}}{\Gamma(l+1)}
	   \int_0^\infty \E^{-t} t^l \left(1+\frac{\I t}{2k}\right)^l dt,
       \quad l>-\frac{1}{2}.
\end{align*}
Therefore, the Jost solution of the unperturbed Bessel equation admits the representation
\[
h_l(k,x) 
	= \E^{-\I kx} f_l(k,x) 
	= \frac{1}{\Gamma(l+1)} \int_0^\infty \E^{-t} t^l 
	\left(1+\frac{\I t}{2kx}\right)^l dt
\]
and
\[
\partial_k h_l(k,x) 
	= \frac{-\I l}{2 \Gamma(l+1)} \frac{1}{k^2 x}\int_0^\infty \E^{-t} t^{l+1}
	  \left(1+\frac{\I t}{2kx}\right)^{l-1} dt.
\]
The last integral converges absolutely whenever $kx\neq 0$. Indeed, since
\[
(a^2 + t^2)^{\frac{l-1}{2}} 
	\le C_l ( a^{l-1} + t^{l-1}),\quad C_l = \begin{cases} 2^{(l-1)/2}, & l\ge 1, \\ 1, & l\in(-1/2,1) \end{cases},
\]
for all $a$, $t>0$,
we  get
\begin{align*}
\abs{\int_0^\infty \E^{-t} t^{l+1} \left(1+\frac{\I t}{2kx}\right)^{l-1} dt}
	&\le \frac{1}{(2|k|x)^{l-1}} 
	     \int _0^\infty  \E^{-t} t^{l+1} (4|k|^2x^2 + t^2)^{\frac{l-1}{2}} dt
	     \\
    &\le \frac{C_l}{(2|k|x)^{l-1}} 
         \left( \Gamma(l+2)(2|k|x)^{l-1} + \Gamma(2l+1)\right).
\end{align*}
Therefore, we end up with the following estimate
\be\label{eq:est.h.l}
\abs{\partial_k h_l(k,x)}
	\le \frac{C}{x|k|^2} \left(\frac{1+|k| x}{|k|x}\right)^{l-1},
\ee
for all $x>0$ and $k\neq 0$.

\begin{remark}
The estimate \eqref{eq:est.h.l} is the best possible. Indeed, if $l\in \N$, then 
\[
h_l(k,x) 
	= \frac{1}{\Gamma(N+1)} 
	  \sum_{j=0}^l \binom{l}{j} \frac{\I^j \Gamma(l+j+1)}{(2x)^j k^j}, 
\]
and
\[
\partial_k h_l(k,x) 
	= \frac{-1}{\Gamma(l+1)} 
	  \sum_{j=1}^l \binom{l}{j} \frac{\I^j j\Gamma(l+j+1)}{(2x)^j k^{(j+1)}}. 
\]
\end{remark}

%%%%%%%%%%%%%%%%%%%%%%%%%%%%%%%%%%%%%%%%%%%%%%%%%%%%%%%%%%%%%%%%%%%%%
%\bigskip
\noindent{\bf Acknowledgments.}
Part of this work was done while J.\ H.\ T.\ visited the Faculty of 
Mathematics of the University of Vienna, in the autumn of 2014. He 
thanks them for their kind hospitality. We also thank the three anonymous
referees for their careful reading of our manuscript leading to several
improvements. In particular this lead to \cite{HKT}.
%%%%%%%%%%%%%%%%%%%%%%%%%%%%%%%%%%%%%%%%%%%%%%%%%%%%%%%%%%%%%%%%%%%%%

%%%%%%%%%%%%%%%%%%%%%%%%%%%%%%%%%%%%%%%%%%%%%%%%%%%%%%%%%%%%%%%%%%%%%

\end{document}